\documentclass[smallextended]{svjour3}

\usepackage{lmodern}
\usepackage[dvipsnames]{xcolor}
\usepackage{amsmath}
\usepackage{amssymb, xfrac, hyperref, enumitem, tikz}
\usepackage{blindtext}
\usepackage[linesnumbered,ruled]{algorithm2e}
\usepackage{csvsimple}
\usepackage{booktabs}
\usepackage{multirow}
\usepackage{graphicx}
\graphicspath{{..}}

\newcommand{\mbf}[1]{\mathbf{#1}}
\newcommand{\mbb}[1]{\mathbb{#1}}

\newcommand{\mcf}[1]{\mathcal{#1}}
\newcommand{\mbfs}[1]{\boldsymbol{#1}}

\newcommand{\replace}{replace}
\newcommand{\drop}{drop}
\newcommand{\conv}{conv}

\newcommand{\pop}{pop}

\makeatother

\title{Compressing Branch-and-Bound Trees\footnote{A short version of this article was accepted for publication at IPCO 2023~\cite{MPX2023}. This extended version contains more detailed discussions and proofs, and new computational contributions and experiments. }}
\author{Gonzalo Mu\~noz$^{\star}$ \and Joseph
Paat$^{\dagger}$ \and \'{A}linson S. Xavier$^{\ddagger}$}
\institute{${\star}$ Institute of Engineering Sciences, Universidad de O'Higgins, Rancagua, Chile\\
\email{gonzalo.munoz@uoh.cl}\\
${\dagger}$ Sauder School of Business, University of British Columbia, Vancouver
BC, Canada,\\
\email{joseph.paat@sauder.ubc.ca}\\
${\ddagger}$ Energy Systems and Infrastructure Analysis Division, Argonne National Laboratory, Lemont, IL, USA\\
\email{axavier@anl.gov}
}

\begin{document}
\maketitle

\begin{abstract}
A branch-and-bound (BB) tree certifies a dual bound on the value of an integer program.
In this work, we introduce the tree compression problem (TCP): \emph{Given a BB tree $T$ that certifies a dual bound, can we obtain a smaller tree with the same (or stronger) bound by either (1) applying a different disjunction at some node in $T$ or (2) removing leaves from $T$?}
We believe such post-hoc analysis of BB trees may assist in identifying helpful general disjunctions in BB algorithms.
We initiate our study by considering computational complexity and limitations of TCP.
We then conduct experiments to evaluate the compressibility of realistic branch-and-bound trees generated by commonly-used branching strategies, using both an exact and a heuristic compression algorithm.
\end{abstract}

\section{Introduction}
Consider an integer linear programming (IP) problem
\begin{equation}\label{eqIP}
\min \{\mbf{c}^\top \mbf{x}:\ \mbf{x} \in \mcf{P} \cap \mbb{Z}^n \},
\end{equation}
where $\mbf{c} \in \mbb{Q}^n$ and
\(
\mcf{P} := \left\{\mbf{x} \in \mbb{R}^n:\ \mbf{A} \mbf{x} \le \mbf{b}\right\}
\)
for $\mbf{A} \in \mbb{Q}^{m\times n}$ and $\mbf{b} \in \mbb{Q}^m$.
Primal bounds on~\eqref{eqIP} can be certified by integer feasible solutions $\mbf{z} \in \mcf{P} \cap \mbb{Z}^n$.
Dual bounds on~\eqref{eqIP}, on the other hand, are typically certified using {\bf branch-and-bound (BB) trees}. 
%
%
A BB tree is a graph-theoretical tree $T$ where each node $v$ corresponds to a polyhedron $\mcf{Q}(v)$, with the root corresponding to $\mcf{P}$.
Moreover, $v$ is either a leaf, or it has exactly two children corresponding to the polyhedra defined by applying a disjunction $(\mbfs{\pi}^\top \mbf{x} \le \pi_0) \lor (\mbfs{\pi}^\top \mbf{x} \ge \pi_0+1)$ to $\mcf{Q}(v)$, where we call $\mbfs{\pi} \in \mbb{Z}^n$ the {\bf branching direction} and $\pi_0 \in \mbb{Z}$.
If we solve the corresponding linear programs over all leaves of $T$, then the smallest value obtained over all leaves yields a dual bound for~\eqref{eqIP}. 
See Section~\ref{secTCP} for a formal definition of BB trees and the dual bound.

In order to generate a BB tree, one must identify a strategy for selecting a leaf of the tree and a strategy for selecting a disjunction to apply. 
See~\cite{LS1999} for a survey on different strategies. 
In practical implementations of the BB method, the only allowed directions are typically $\{\mbf{e}^1,\dotsc, \mbf{e}^n\}$, in which case we say the algorithm uses {\it variable disjunctions}.
However, many results explore the benefit of additional directions: 
various subsets of $\{-1,0,1\}^n$ are explored in~\cite{MC2013,OM2001,YBS2021}; directions derived from mixed integer Gomory cuts are explored in~\cite{CLN2011,KC2011}; directions derived using basis reduction techniques are explored in~\cite{AL2004,ML2010}; Mahajan and Ralphs~\cite{MR2009} solve a subproblem to find a disjunction that closes the duality gap by a certain amount.
The largest set of directions is the set $\mbb{Z}^n$, in which case the algorithm uses {\it general disjunctions}.

Although a larger set of allowable directions provides more flexibility, it has been repeatedly verified that searching through this set during the execution of the algorithm can be computationally expensive~\cite{GMBGS2015,MR2009}.
The work in this paper follows a different approach to identify meaningful directions. 
Given a tree $T$ produced using some set of allowable directions $\mcf{D} \subseteq \mbb{Z}^n$, we ask if $T$ can be ``compressed'' into a smaller tree with the same (or stronger) dual bound by using a potentially larger set of directions $\mcf{D}' \supseteq \mcf{D}$, and a limited set of transformations.
This post-hoc compression analysis is more restricted and allows one to use a global view of the tree to identify potentially meaningful branching directions, as opposed to the dynamic approach.

The motivation behind this compression question comes from recent trends to better understand BB trees, in particular, to study how hard it is to generate small trees \cite{GM2022}, how big the trees produced by certain branching rules can be \cite{DDMS2021}, and how we can learn good but expensive branching disjunctions \cite{gasse2019exact}.
We believe that by successfully compressing a BB tree that was produced by state-of-the-art methods, we can (a) find practical ways of producing small trees to be used as effective dual certificates \cite{cheung2017verifying}, (b) identify strong general disjunctions for a family of instances and (c) produce training data for learn-to-branch strategies. With this in mind, we focus our work on the theory and practical approaches to compressing branching trees.

\smallskip
\noindent{\bf Related work.}
To the best of our knowledge, this is the first piece of work to study the tree compression problem.
A related question is the minimum size of a BB tree certifying optimality or infeasibility of~\eqref{eqIP}; we use some of these results in our own work.
%
%
Chv\'{a}tal~\cite{C1980} and Jeroslow~\cite{J1974} give examples of IPs that require a BB tree whose size is exponential in the number of variables $n$ when only variable directions $ \{\mbf{e}^1,\dotsc, \mbf{e}^n\}$ are used to generate disjunctions.
There are examples where an exponential lower bound in $n$ cannot be avoided even with general disjunctions~\cite{DT2020,DDM2022}.
Basu et al.~\cite{BCDSJ2021} consider the set $\mcf{D}_s$ of directions whose support is at most $s$; they prove that if $s \in O(1)$, then a BB tree proving infeasibility of Jeroslow's instance has exponential in $n$ many nodes~\cite{BCDSJ2021}.
For an interesting perspective on provable upper bounds, Dey et al.~\cite{DDMS2021} relate the size of BB trees generated using full strong branching and variable disjunctions to the additive integrality gap for certain classes of instances like vertex cover. 

Pfetsch et al.~\cite{GM2022} show that it is NP-hard to find the smallest BB tree generated using only variable disjunctions.
Mahajan and Ralphs~\cite{MR2010} show that it is NP-complete to decide whether there exists a general disjunction proving infeasibility at the root node. 
They also provide a MIP that can be solved at a node in a BB tree to yield a disjunction maximizing the dual bound improvement.

The tree compression problem is a post-hoc analysis of a BB tree. 
A similar kind of analysis is done in backdoor branching, where one explores a tree $T$ to find small paths from the root to the optimal solution with the ultimate aim to identify good branching decisions to make next time the algorithm is run on a similar IP~\cite{FM2012,KVD2022}.
The major difference between backdoor branching and the compression question is that the former only considers finding a path in a tree while the latter considers how to modify a tree to create short paths.
Another form of post-hoc analysis is tree balancing, where the goal is to transform a tree $T$ proving integer infeasibility into a new tree with the same dual bound whose size is polynomial in $|T|$ and whose depth is polylogarithmic in $|T|$; see, e.g.,~\cite{BNIKPPR20118} for a discussion on balancing and stabbing planes.
A major difference between the balancing question and the compression question is that the former is allowed to grow the tree along branches while the latter is not.

\medskip

\noindent{\bf Contributions.}
We introduce the tree compression problem in Section~\ref{secTCP}.
In Theorem~\ref{thNPComplete}, we show that the problem is NP-Complete when $\mcf{D} = \mbb{Z}^n$ and $\mbf{c} = \mbf{0}$.
We then demonstrate in Theorem~\ref{thLowerBound} that tree compression does not always give the smallest BB tree meeting a certain dual bound. 
In fact, we give an example of a BB tree $T$ of size $|T| \ge 2^{n+1}-1$ that cannot be compressed to a BB tree with fewer than $\sfrac{(2^n-1)}{n}$ nodes, yet there is a different BB tree with the same root and dual bound with only $7$ nodes.
These results appear in Section~\ref{secTheory}.

From a more practical standpoint, we also provide extensive computational results on the compression problem.
We first look at BB trees from MIPLIB 3.0~\cite{BBI1992} instances generated using \emph{full strong branching}, the state-of-the-art variable branching strategy with respect to tree size, and \emph{reliability branching with plunging}, often considered the state-of-the-art branching strategy with respect to running time.
We first compress these trees using a computationally-expensive exact algorithm based on a MIP formulation by Mahajan and Ralphs~\cite{MR2009,MR2010}.
We then evaluate how much of this compression is achievable in a short amount of time, by applying a heuristic algorithm based on the iterative procedure introduced by Owen and Mehrota~\cite{OM2001}.
Overall, we see that many MIPLIB 3.0 trees can be significantly compressed.
Moreover, we find that the heuristic procedure achieves good compression.
These algorithms and results are described in Sections~\ref{secMethods} and ~\ref{secExperiments}, respectively.

Finally, we consider the more challenging instances of MIPLIB 2017 \cite{miplib2017}. We propose various node processing rules in the compression heuristic and show via extensive computational experiments that some strategies can produce considerably smaller trees in moderate running times. These experiments are described in Section \ref{sec:miplib2017}.

\section{The tree compression problem (TCP)}\label{secTCP}

We define a {\bf branch-and-bound (BB) tree} as a graph-theoretical rooted tree where each node $v$ corresponds to a polyhedron $\mcf{Q}(v)$, and the root node $r$ corresponds to $\mcf{Q}(r) = \mcf{P}$.
Furthermore, each node $v$ is either a leaf, or it has exactly two children corresponding to the polyhedra
\begin{equation}\label{eqDisjunct}
\mcf{Q}(v) \cap \{\mbf{x} \in \mbb{R}^n:\ \mbfs{\pi}^\top \mbf{x} \le \pi_0\}
\quad\text{and}\quad
\mcf{Q}(v) \cap \{\mbf{x} \in \mbb{R}^n:\ \mbfs{\pi}^\top \mbf{x} \ge \pi_0+1\},
\end{equation}
where $\mbfs{\pi} \in \mbb{Z}^n$ is called the {\bf branching direction} and $\pi_0 \in \mbb{Z}$.
The {\bf dual bound} relative to $\mbf{c} \in \mbb{Q}^n$ provided by a BB tree $T$ is
\[
d(T, \mbf{c}) := \min_{ v \in L(T)}\ \min \{\mbf{c}^\top \mbf{x} :\ \mbf{x} \in \mcf{Q}(v)\},
\]
where $L(T)$ is the set of leaves of $T$.
If $\mcf{Q}(v) = \emptyset$ for some $v \in L(T)$, then set $ \min \{\mbf{c}^\top \mbf{x} :\ \mbf{x} \in \mcf{Q}(v)\} := \infty$.
Define $d(T, \mbf{c}) = \infty$ if $\mcf{Q}(v) = \emptyset$ for each $v \in L(T)$, and $d(T, \mbf{c}) = -\infty$ if $\mbf{x} \mapsto \mbf{c}^\top \mbf{x}$ is unbounded from below over $\mcf{Q}(v)$ for some $v \in L(T)$.
For simplicity, our definition allows BB trees that have multiple nodes corresponding to the same polyhedron, although such trees would typically not be generated by well-designed BB algorithms.
We also do not require the tree to certify infeasibility or optimality of~\eqref{eqIP}; this allows for trees generated by partial (e.g. time- or node-limited) runs of the BB method.

Let $T$ be a BB tree and $v \in T$ be a non-leaf node.
Our notion of compression is based on two operations on $T$.
For $(\mbfs{\pi}, \pi_0) \in \mbb{Z}^n\times \mbb{Z}$, let
\[
\replace(T,v,\mbfs{\pi}, \pi_0)
\]
denote the BB tree obtained from $T$ by replacing all descendants of $v$ with the two new children defined by applying the disjunction $(\mbfs{\pi}^\top \mbf{x} \le \pi_0) \lor (\mbfs{\pi}^\top \mbf{x} \ge \pi_0+1)$ to $\mcf{Q}(v)$, i.e., the two new children are the polyhedra in~\eqref{eqDisjunct}.
We use 
\[
\drop(T,v)
\]
to denote the BB tree obtained from $T$ by removing all descendants of $v$.

We refer to the number of nodes in $T$ as the {\bf size} of $T$ and denote it by $|T|$.
A BB tree $T'$ is a {\bf compression of $T$} if there exists a sequence of BB trees $T_1 = T, T_2, \dotsc, T_k = T'$ such that for each $i \in \{2, \dotsc, k\}$ we have
\begin{enumerate}
    \item\label{compress1} Either $T_i = \drop(T_{i-1},v)$ for some $v \in T_{i-1}$, or $T_i = \replace(T_{i-1},v, \mbfs{\pi}, \pi_0)$ for some $v \in T_{i-1}$ and $(\mbfs{\pi}, \pi_0) \in \mbb{Z}^n\times \mbb{Z}$.
    
    \smallskip
    \item\label{compress2} $|T_i| < |T_{i-1}|$ and $d(T_i, \mbf{c}) \ge d(T_{i-1}, \mbf{c})$.
\end{enumerate}
The definition of compression depends on the dual bound of $T$. 
Also, observe that the replacement operation only acts on non-leaf nodes and thus only produces children of non-leaf nodes.
Consequently, leaf nodes of a BB tree will either remain leaf nodes or disappear from the tree during the compression process.
Given that the replacement operation creates two new nodes that are leaves themselves, the previous discussion implies that any new 
disjunctions introduced in the compression process appear near the bottom of the tree.

As an example of these definitions, consider $\mcf{P} := [0,\sfrac{1}{5}]^2$ and the BB tree $T$ depicted in Figure \ref{fig:bbtree}.
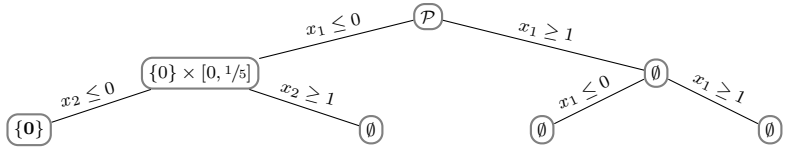
\begin{figure}
\begin{center}
    \begin{tikzpicture}[scale = .75, every node/.style={scale=0.85}]
    \node(p)[draw = black!50, rectangle, rounded corners, thick] at (0,0){$\mcf{P}$};
    \node(q1)[draw = black!50, rectangle, rounded corners, thick] at (-4,-1){$\{0\}\times [0,\sfrac{1}{5}]$};
    \node(q2)[draw = black!50, rectangle, rounded corners, thick] at (4,-1){$\emptyset$};
    \node(q3)[draw = black!50, rectangle, rounded corners, thick] at (-7,-2){$\{\mbf{0}\}$};
    \node(q4)[draw = black!50, rectangle, rounded corners, thick] at (-1,-2){$\emptyset$};
    \node(q5)[draw = black!50, rectangle, rounded corners, thick] at (2,-2){$\emptyset$};
    \node(q6)[draw = black!50, rectangle, rounded corners, thick] at (6,-2){$\emptyset$};
    \draw(p) to node[rotate = 12, above, pos = .5]{$x_1 \le 0$}(q1);
    \draw(p) to node[rotate = -12, above, pos = .5]{$x_1 \ge 1$}(q2);
    \draw(q1) to node[rotate = 15, above, pos = .6]{$x_2 \le 0$}(q3);
    \draw(q1) to node[rotate = -15, above, pos = .5]{$x_2 \ge 1$}(q4);
    \draw(q2) to node[rotate = 25, above, pos = .6]{$x_1 \le 0$}(q5);
    \draw(q2) to node[rotate = -25, above, pos = .5]{$x_1 \ge 1$}(q6);
    \end{tikzpicture}
    \end{center}
    \caption{Example of BB tree. Here, $\mcf{P} := [0,\sfrac{1}{5}]^2$, disjunctions are indicated on edges and polyhedra in the nodes.} \label{fig:bbtree}
\end{figure}
Note that we allow a BB tree to have disjunctions at empty nodes, and disjunctions may be repeated. 
Let $\mbf{c} = (-1,-1)$; we have $d(T,\mbf{c}) =0 $.
We can compress $T$ with the drop operation at the right child $v_2$ of the root $r$; see Figure \ref{fig:bbtreecompressed}(a).
We can also compress $T$ with the replace operation at the root with $\mbfs{\pi} = -\mbf{c}$ and $\pi_0 =0 $; see Figure~\ref{fig:bbtreecompressed}(b).
It can be checked that $d(\drop(T, v_2), \mbf{c}) = d(\replace(T, r, \mbfs{\pi}, 0), \mbf{c}) = 0$.

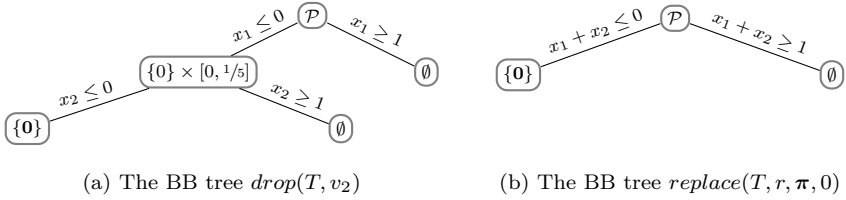
\begin{figure}
\begin{center}
\begin{tabular}{c@{\hskip .75 cm}c}
    \begin{tikzpicture}[scale = .75, every node/.style={scale=0.85}]
    \node(p)[draw = black!50, rectangle, rounded corners, thick] at (0,0){$\mcf{P}$};
    \node(q1)[draw = black!50, rectangle, rounded corners, thick] at (-2,-1){$\{0\}\times [0,\sfrac{1}{5}]$};
    \node(q2)[draw = black!50, rectangle, rounded corners, thick] at (2,-1){$\emptyset$};
    \node(q3)[draw = black!50, rectangle, rounded corners, thick] at (-5,-2){$\{\mbf{0}\}$};
    \node(q4)[draw = black!50, rectangle, rounded corners, thick] at (.5,-2){$\emptyset$};
    \draw(p) to node[rotate = 25, above, pos = .5]{$x_1 \le 0$}(q1);
    \draw(p) to node[rotate = -25, above, pos = .5]{$x_1 \ge 1$}(q2);
    \draw(q1) to node[rotate = 15, above, pos = .6]{$x_2 \le 0$}(q3);
    \draw(q1) to node[rotate = -20, above, pos = .6]{$x_2 \ge 1$}(q4);
    \end{tikzpicture}
    &
    \begin{tikzpicture}[scale = .75, every node/.style={scale=0.85}]
    \node(p)[draw = black!50, rectangle, rounded corners, thick] at (0,0){$\mcf{P}$};
    \node(q1)[draw = black!50, rectangle, rounded corners, thick] at (-2.75,-1){$\{\mbf{0}\}$};
    \node(q2)[draw = black!50, rectangle, rounded corners, thick] at (2.75,-1){$\emptyset$};
    \node(q4)[draw = none, rectangle, rounded corners, thick] at (0,-2){\phantom{$\emptyset$}};
    \draw(p) to node[rotate = 20, above, pos = .5]{$x_1+x_2 \le 0$}(q1);
    \draw(p) to node[rotate = -20, above, pos = .5]{$x_1+x_2 \ge 1$}(q2);
    \end{tikzpicture}\\[.25 cm]
    (a) The BB tree $\drop(T, v_2)$
    &
    (b) The BB tree $\replace(T, r, \mbfs{\pi}, 0)$
    \end{tabular}
\end{center}
\caption{Example of different compressions of the tree $T$ in Figure \ref{fig:bbtree}.} \label{fig:bbtreecompressed}
\end{figure}

For an example of an invalid compression operation, consider replacing $\sfrac{1}{5}$ in the original example by $\sfrac{1}{2}$.
After this replacement, $\replace(T, r, \mbfs{\pi}, 0)$ would no longer be a compression because it would deteriorate the lower bound to $d(\replace(T, r, \mbfs{\pi}, 0), \mbf{c}) = -1$; the rightmost branch of Figure~\ref{fig:bbtreecompressed} (b) would not be empty, as it would contain the point $(\sfrac{1}{2},\sfrac{1}{2})$ which is the optimal solution for that node.

The original example illustrates that strict dual improvement is not necessary in the compression process.
However, it is possible for the dual bound to improve during the compression process. 
For instance, consider replacing $\mcf{P}$ in the example with the triangle with vertices $(-\sfrac{1}{2},-\sfrac{1}{2})$, $(-\sfrac{1}{2},1)$, $(1, -\sfrac{1}{2})$; still use $\mbf{c} = (-1,-1)$.
If we use the same disjunctions as in Figure \ref{fig:bbtree}, then we produce a BB tree $T'$, which has a dual bound of $-\sfrac{1}{2}$, obtained at $(\sfrac{-1}{2},1)$ and $(1, -\sfrac{1}{2})$.
One can also check that $d(\replace(T', r, \mbfs{\pi}, 0), \mbf{c}) = 0$. 
Hence, compression can improve the dual bound.

We now formally define the compression problem.

\begin{definition}
The {\bf tree compression problem (TCP)} with respect to a set of allowable directions $\mcf{D}$ is defined as follows:
{\it Given a BB tree $T$ and an objective vector $\mbf{c} \in \mbb{Q}^n$, is there a compression of $T$ where the replacement operation only uses branching directions in $\mcf{D}$?}
\end{definition}

There is an optimization version of this question in which we try to compress $T$ as much as possible. 
Section~\ref{secTheory} considers the decision problem (showing this is NP-Complete) and the optimization problem (showing limitations of compression). 
Our computational results in Sections~\ref{secMethods}, ~\ref{secExperiments} and \ref{sec:miplib2017} consider the optimization problem.
As seen in the previous example, the choice of $\mcf{D}$ influences the compression question; the BB tree in Figure \ref{fig:bbtreecompressed}(a) is the best compression if $\mcf{D}$ only contains unit vectors while the BB tree in Figure \ref{fig:bbtreecompressed}(b) is the best compression if $\mcf{D}$ contains the all-ones vector.

\section{Complexity results and lower bounds}\label{secTheory}

In this section, we show (TCP) is NP-Complete when $\mcf{D} = \mbb{Z}^n$ and $\mbf{c} = \mbf{0}$.
Our proof uses a reduction from the following problem.

\begin{definition}
The problem of {\bf disjunctive infeasibility (DI)} is defined as follows:
{\it Given $\mbf{A} \in \mbb{Q}^{m\times n}$ and $\mbf{b} \in \mbb{Q}^m$ that define a polyhedron $\mcf{S} = \{\mbf{x}\in \mbb{R}^n: \mbf{A}\mbf{x}\le\mbf{b}\}$, decide if there exists $\mbfs{\pi} \in \mbb{Z}^n\setminus\{\mbf{0}\}$ and $\pi_0 \in \mbb{Z}$ such that}
\[
\mcf{S} \subseteq \{\mbf{x} \in \mbb{R}^n: \pi_0 < \mbfs{\pi}{}^\top \mbf{x} < \pi_0+1\}.
\]
\end{definition}
(DI) was proved to be NP-Complete in \cite[Proposition 3.2]{MR2010}.
Keep in mind that the input to (DI) is a single polyhedron whereas the input to (TCP) is an entire BB tree.
Also note that, although the two problems are related, (DI) considers operating on a single polyhedron, whereas the compression can be accomplished at any node in the BB tree.

\begin{theorem}\label{thNPComplete}
{\rm (TCP)} is NP-Complete when $\mcf{D} = \mbb{Z}^n$ and $\mbf{c} = \mbf{0}$.
\end{theorem}

Before moving to the proof, we note that  
(TCP) can be answered in polynomial time if the set $\mcf{D}$ of directions allowed in the replacement operation is finite and polynomial in the encoding size of $T$, e.g., $\mcf{D} = \{\mbf{e}^1,\dotsc, \mbf{e}^n\}$.
Indeed, one can try the drop operation at each node and the replace operation for each node-direction pair $(v, \mbf{d})$; this requires polynomial time due to the size of $\mcf{D}$.

\begin{proof}
We first argue that (TCP) is in NP when $\mcf{D} = \mbb{Z}^n$ and $\mbf{c} = \mbf{0}$.
Let $T$ be a BB tree that can be compressed. 
Either $d(T, \mbf{0}) = 0$, which happens if $\mcf{Q}(v) \neq \emptyset$ for some $v \in L(T)$, or $d(T, \mbf{0}) = \infty$, which happens if $\mcf{Q}(v)= \emptyset$ for all $v \in L(T)$.
We need to argue that there is a certificate (in the form of a BB tree $T'$) whose encoding size is polynomially bounded by the encoding size of $T$; checking that $T'$ is a compression can be done by checking feasibility of the leaves of $T'$, which can be done in polynomial time as each problem is a linear program. 

Suppose $d(T, \mbf{0}) = 0$.
Then the BB tree $T' = \drop(T,r)$ is non-empty because the assumption $d(T, \mbf{0}) = 0$ implies $\mcf{Q}(v) \neq \emptyset$ for some $v \in L(T)$.
Moreover, because $T$ is compressible and compression can only happen at non-leaf nodes,
it must be the case that $T$ consists of more than just the root $r$.
Hence, $|T'| < |T|$ and $d(T, \mbf{0})=d(T', \mbf{0}) = 0$.
Thus, $T'$ is a certificate of compressibility.

Suppose that $d(T, \mbf{0}) = \infty$.
Thus, $\mcf{Q}(v)= \emptyset$ for all $v \in L(T)$.
If a non-leaf node $v$ of $T$ satisfies $\mcf{Q}(v) = \emptyset$, then $T' = \drop(T,v)$ is a compression of $T$ whose size is polynomial in the size of $T$.
So, suppose that $\mcf{Q}(v) \neq \emptyset$ for all non-leaf nodes of $T$.
Since $T$ can be compressed (and the drop operation cannot be used from the previous sentences), there exists a non-leaf $v \in T$ and $(\mbfs{\pi}, \pi_0) \in \mbb{Z}^n\times \mbb{Z}$ such that applying the disjunction $(\mbfs{\pi}^\top \mbf{x} \le \pi_0) \lor (\mbfs{\pi}^\top \mbf{x} \ge \pi_0+1)$ to $\mcf{Q}(v)$ will yield two empty polyhedra.
In demonstrating that (DI) is in NP, Mahajan and Ralphs prove that $(\mbfs{\pi}, \pi_0)$ can be chosen to have encoding size polynomial in the encoding size of $\mcf{Q}(v)$~\cite[\S 3]{MR2010}.
Hence, there is a compression $T' = \replace(T, v,\mbfs{\pi}, \pi_0)$ of $T$ whose encoding size is polynomial in the encoding size of $T$.
This shows that (TCP) is in NP when $\mcf{D} = \mbb{Z}^n$ and $\mbf{c} = \mbf{0}$.\\

We now proceed to reduce (DI) to (TCP). Consider an instance $(\mbf{A}, \mbf{b})$ of (DI). 
Let $\mbf{x}^* \in \mcf{S} \setminus \mbb{Z}^n$; this can be found in polynomial time unless $\mcf{S}$ is empty (in which case the answer to (DI) is `yes') or a single integer vector (in which case the answer is `no').
Without loss of generality, $x^*_1 \not \in \mbb{Z}$.

We lift $\mcf{S}$ into $\mbb{R}^{n+1}$ to create an instance of (TCP).
We write a point in $\mbb{R}^{n+1}$ as $(\mbf{x}, y) \in \mbb{R}^n\times \mbb{R}$.
Define
\[
\mcf{P} := \conv \left(\left\{(\mbf{x}^*,0), (\mbf{x}^*,1)\right\} \cup \left\{(\mbf{x}, \sfrac{1}{2}):\ \mbf{x} \in \mcf{S}\right\}\right)
\]
We build a BB tree $T$ with root node $r$ and $\mcf{Q}(r) = \mcf{P}$.
Branch on the disjunction $(y \le 0) \lor (y \ge 1)$ at $r$ to obtain $v_1$ and $v_2$:
\[
\begin{array}{rclcl}
\mcf{Q}(v_1) & := & \{(\mbf{x}, y) \in \mcf{P}:\ y \le 0 \} &=& \{(\mbf{x}^*, 0)\}\\
\mcf{Q}(v_2) & := & \{(\mbf{x}, y) \in \mcf{P}:\ y \ge 1 \} &=& \{(\mbf{x}^*, 1)\}.
\end{array}
\]
Branch on $v_1$ and $v_2$ using $(x_1 \le \lfloor x^*_1\rfloor) \lor (x_1 \ge \lceil x^*_1\rceil)$ to obtain $v_3,v_4, v_5, v_6$: 
\[
\begin{array}{rclcl}
\mcf{Q}(v_3) & := & \{(\mbf{x}, y) \in \mcf{P}:\ y \le 0~\text{and}~x_1 \le \lfloor x^*_1 \rfloor\} &=& \emptyset\\
\mcf{Q}(v_4) & := & \{(\mbf{x}, y) \in \mcf{P}:\ y \le 0~\text{and}~x_1 \ge \lceil x^*_1\rceil \} &=& \emptyset\\
\mcf{Q}(v_5) & := & \{(\mbf{x}, y) \in \mcf{P}:\ y \ge 1~\text{and}~x_1 \le \lfloor x^*_1 \rfloor\} &=& \emptyset\\
\mcf{Q}(v_6) & := & \{(\mbf{x}, y) \in \mcf{P}:\ y \ge 1~\text{and}~x_1 \ge \lceil x^*_1\rceil \} &=& \emptyset.
\end{array}
\]
$T$ has $7$ nodes, and the polyhedra corresponding to the four leaves $v_3, v_4, v_5, v_6$ are empty.
The encoding size of $T$ is polynomial in the encoding size of $\mcf{S}$.
%
%

If (DI) has a `yes' answer with certificate $\mbfs{\pi} \in \mbb{Z}^n \setminus \{\mbf{0}\}$ and $\pi_0 \in \mbb{Z}$, then
\[
\mcf{P} \subseteq \mcf{S}\times \mbb{R} \subseteq \{(\mbf{x}, y)\in \mbb{R}^n \times \mbb{R}: \pi_0 < \mbfs{\pi}^\top \mbf{x}< \pi_0+1\}.
\]
%
Hence, the answer to (TCP) is `yes' because $\replace(T,r, (\mbfs{\pi},0), \pi_0)$ is a compression of $T$.
Assume (TCP) has a `yes' answer.
%
%
The drop operation can only be applied to $r, v_1$ or $v_2$, and doing so to any of these does not compress the tree because the dual bound decreases. 
So, the `yes' answer must come from the replace operation.
In order to decrease the size of the tree, which is required for compression, the replace operation must be applied at $r$.
Therefore, there is a non-zero $(\mbfs{\pi}, \pi_{n+1}) \in \mbb{Z}^n\times \mbb{Z}$ and some $\pi_0 \in \mbb{Z}$ such that 
\[\pi_0 <  \mbfs{\pi}^\top \mbf{x} + \pi_{n+1}y < \pi_0+1 \text{ for all } (\mbf{x}, y) \in \mcf{P}.\]
Note that $\mbfs{\pi} \neq \mbf{0}$ and $\pi_{n+1} = 0$ as otherwise $(\mbf{x}^*, 0)$ or $(\mbf{x}^*, 1)$ violates one of these inequalities.
The tuple $(\mbfs{\pi}, \pi_0)$ provides a `yes' answer to (DI). 
\qed
\end{proof}

Our next result is motivated by recent studies of the minimal sizes of BB trees, e.g. \cite{DDMS2021,PSW2022}.
In the following theorem, we show that tree compression does not always yield the smallest tree for a given dual bound.

\begin{theorem}\label{thLowerBound}
Let $\mcf{D} = \mbb{Z}^n$. For $n \ge 2$, there exists a polytope $\mcf{P} \subseteq \mbb{R}^{n+1}$ and a BB tree $T$ with root polyhedron $\mcf{P}$ such that 
\begin{enumerate}
    \item\label{item2} $|T|\ge 2^{n+1}-1$ and $d(T, \mbf{0}) = \infty$.
    \item\label{item3} $T$ cannot be compressed to a tree with fewer than $ \sfrac{(2^n-1)}{n}$ nodes.
    \item\label{item4} There exists a tree $T'$ with root $\mcf{P}$, $|T'| = 7$ and $d(T, \mbf{0}) = d(T', \mbf{0})$.  
\end{enumerate}
\end{theorem}

\begin{proof}
Let $\overline{\mcf{P}} \subseteq [0,1]^n$ be a polytope satisfying $\overline{\mcf{P}} \cap \mbb{Z}^n = \emptyset$ and if a tree $\overline{T}$ with root $\overline{\mcf{P}}$ satisfies $d(\overline{T}, \mbf{0}) = \infty$, then $|\overline{T}| \ge 2^{n+1}-1$.
One such $\overline{\mcf{P}}$ comes from~\cite[Proposition 3]{DDM2022}.
Let $\overline{T}$ be a BB tree of minimal size with root $\overline{\mcf{P}}$ and $d(\overline{T}, \mbf{0}) = \infty$.
We will manipulate $\overline{T}$ and $\overline{\mcf{P}}$ to build the desired $\mcf{P}$, $T$ and $T'$.

The minimality of $\overline{T}$ and $d(\overline{T}, \mbf{0}) = \infty$ implies that a node $\overline{v} \in \overline{T}$ satisfies $\mcf{Q}(\overline{v}) = \emptyset$ if and only if $\overline{v} \in L(\overline{T})$.
Consider a non-leaf node $\overline{v} \in \overline{T}$; from the previous sentence, $\mcf{Q}(\overline{v}) \neq \emptyset$.
Moreover, given that $\overline{\mcf{P}} \cap \mbb{Z}^n = \emptyset$, the polyhedron $\mcf{Q}(\overline{v})$ is integer infeasible, i.e., $\mcf{Q}(\overline{v})\setminus  \mbb{Z}^n = \mcf{Q}(\overline{v}) \neq \emptyset$.
Putting all of this together with the assumption that $\overline{\mcf{P}} \subseteq [0,1]^n$, we can conclude that there exists an index $i_{\overline{v}} \in \{1, \dotsc, n\}$ and a point $\mbf{x} \in \mcf{Q}(\overline{v})$ with $x_{i_{\overline{v}}} \in (0,1)$.

There exist $\sfrac{(|\overline{T}|-1)}{2} \ge 2^n-1$ non-leaf nodes in $\overline{T}$.
Therefore, there exists an index $i^* \in \{1, \dotsc, n\}$ such that at least $\sfrac{(2^n-1)}{n}$ nodes $\overline{v} \in \overline{T}$ have some point $\mbf{x} \in \mcf{Q}(\overline{v})$ with $x_{i^*} \in (0,1)$.
We denote the set of these nodes as
\[
\overline{N} := \{\overline{v} \in \overline{T}:\ \exists~\mbf{x} \in \mcf{Q}(\overline{v}) ~\text{with}~x_{i^*} \in (0,1)\}.
\]
For each $\overline{v}  \in \overline{N}$, arbitrarily choose a point in $\mcf{Q}(\overline{v})$ whose ${i^*}$th component is in $(0,1)$ and call this point $\mbf{x}(\overline{v} )$.
Define
\[
\mcf{P} := \conv \left( \left\{\vphantom{\frac{1}{2}}(\mbf{x}(\overline{v} ), t):\ \overline{v}  \in \overline{N}~\text{and}~t \in \{0,1\}\right\}
\cup \left(\overline{\mcf{P}}\times \left\{\frac{1}{2}\right\}\right)\right).
\]
Note that $\mcf{P} \cap \mbb{Z}^{n+1} = \emptyset$.

We create a BB tree $T'$ with root polyhedron $\mcf{P}$ and $d(T', \mbf{0}) = \infty$ by first branching on $(x_{n+1} \le 0) \lor (x_{n+1} \ge 1)$; the polyhedra of the resulting children are $\conv\{(\mbf{x}(\overline{v}), t): \overline{v} \in \overline{N}\}$ for $t \in \{0,1\}$.
Given that $x(\overline{v})_{i^*} \in (0,1)$ for each $\overline{v}\in \overline{N}$, we can branch on each $\conv\{(\mbf{x}(\overline{v}), t):\ \overline{v} \in \overline{N}\}$ using $(x_{i^*} \le 0) \lor (x_{i^*} \ge 1)$ to obtain all empty children nodes. 
We illustrate tree $T'$ in Figure \ref{fig:size7tree}.
This proves~\ref{item4}.

\begin{figure}
\begin{center}
    \begin{tikzpicture}[scale = .75, every node/.style={scale=0.85}]
    \node(p)[draw = black!50, rectangle, rounded corners, thick] at (0,0){$\mcf{P}$};
    \node(q1)[draw = black!50, rectangle, rounded corners, thick] at (-4,-1){$\conv\{(\mbf{x}(\overline{v}), 0): \overline{v} \in \overline{N}\}$};
    \node(q2)[draw = black!50, rectangle, rounded corners, thick] at (4,-1){$\conv\{(\mbf{x}(\overline{v}), 1): \overline{v} \in \overline{N}\}$};
    \node(q3)[draw = black!50, rectangle, rounded corners, thick] at (-6,-2.5){$\emptyset$};
    \node(q4)[draw = black!50, rectangle, rounded corners, thick] at (-2,-2.5){$\emptyset$};
    \node(q5)[draw = black!50, rectangle, rounded corners, thick] at (2,-2.5){$\emptyset$};
    \node(q6)[draw = black!50, rectangle, rounded corners, thick] at (6,-2.5){$\emptyset$};
    \draw(p) to node[rotate = 12, above, pos = .5]{$x_{n+1} \le 0$}(q1);
    \draw(p) to node[rotate = -12, above, pos = .5]{$x_{n+1} \ge 1$}(q2);
    \draw(q1) to node[rotate = 30, above, pos = .6]{$x_{i^*} \le 0$}(q3);
    \draw(q1) to node[rotate = -30, above, pos = .6]{$x_{i^*} \ge 1$}(q4);
    \draw(q2) to node[rotate = 30, above, pos = .6]{$x_{i^*} \le 0$}(q5);
    \draw(q2) to node[rotate = -30, above, pos = .6]{$x_{i^*} \ge 1$}(q6);
    \end{tikzpicture}
    \end{center}
    \caption{Branch and bound tree $T'$ constructed on the proof of Theorem \ref{thLowerBound}} \label{fig:size7tree}
\end{figure}
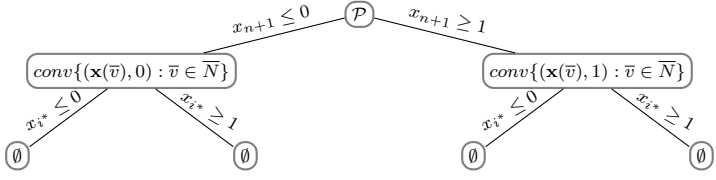

We define $T$ in the theorem by lifting $\overline{T}$.
More precisely, extend every disjunction $(\overline{\mbfs{\pi}}^\top \mbf{x} \le \pi_0) \lor (\overline{\mbfs{\pi}}^\top \mbf{x} \ge \pi_0+1)$ in $\overline{T}$ to a disjunction $(\mbfs{\pi}^\top \mbf{x} \le \pi_0) \lor (\mbfs{\pi}^\top \mbf{x} \ge \pi_0+1)$, where $\mbfs{\pi} := (\overline{\mbfs{\pi}}, 0)$.
Thus, $|T| = |\overline{T}| \ge 2^{n+1}-1$.
Furthermore, $\mcf{P} \subseteq \overline{\mcf{P}} \times \mbb{R}$, so $d(T, \mbf{0}) = \infty$ because $d(\overline{T}, \mbf{0}) = \infty$.
Thus, $T$ satisfies~\ref{item2}.

It remains to prove~\ref{item3}, i.e., that $T$ cannot be significantly compressed.
Assume that $T$ can be compressed via the drop operation.
The corresponding node in $\overline{T}$ can also be dropped.
However, this contradicts the minimality of $\overline{T}$.

We claim that if $v\in T$ corresponds to a node $\overline{v} \in \overline{N} $, then $T$ cannot be compressed at $v$ using the replace operation.
By contradiction, suppose there exists $v \in T$ corresponding to a node $\overline{v} \in \overline{N}$ and a disjunction $(\mbfs{\pi}^\top \mbf{x}+\pi_{n+1}x_{n+1} \le \pi_0) \lor (\mbfs{\pi}^\top \mbf{x}+\pi_{n+1}x_{n+1} \ge \pi_0+1)$ that we can use to compress $T$ at  $v$ via the replace operation, i.e., 
\[\mcf{Q}(v) \subseteq \{(\mbf{x}, \alpha) \in \mbb{R}^n\times \mbb{R}: \pi_0 < \mbfs{\pi}^\top \mbf{x}+\pi_{n+1}\alpha < \pi_0+1\} .\]

If $\pi_{n+1} = 0$, then this disjunction can be projected to $\overline{T}$ to compress it, contradicting the minimality of $\overline{T}$.
Therefore, $\pi_{n+1} \neq 0$.
By the construction of lifting of $\overline{T}$ to create $T$, we guarantee $(\mbf{x}(\overline{v}), \alpha) \in \mcf{Q}(v)$ for each $\alpha \in [0,1]$.
Hence, for each $\alpha \in [0,1]$, the point $(\mbf{x}(\overline{v}), \alpha)$ satisfies 
\(
\pi_0 - \mbfs{\pi}^\top \mbf{x}(\overline{v}) < \pi_{n+1}\alpha < \pi_0 - \mbfs{\pi}^\top \mbf{x}(\overline{v}) +1.
\) 
In particular, if we plug in $\alpha = 0$ and $\alpha = 1$, then we see that $\pi_0 - \mbfs{\pi}^\top \mbf{x}(\overline{v}) < \min \{0, \pi_{n+1}\} $, and $\max\{0,\pi_{n+1}\} < \pi_0 - \mbfs{\pi}^\top \mbf{x}(\overline{v}) +1 $.
If $\pi_{n+1} < 0$, then $\pi_{n+1} \le -1$; hence, $\pi_0 - \mbfs{\pi}^\top \mbf{x}(\overline{v}) \le -1$ and $-1 < \pi_0 - \mbfs{\pi}^\top \mbf{x}(\overline{v})$, which is a contradiction. 
If $\pi_{n+1} > 0$, then $\pi_{n+1} \ge 1$; hence, $\pi_0 - \mbfs{\pi}^\top \mbf{x}(\overline{v}) < 0$ and $0 \le \pi_0 - \mbfs{\pi}^\top \mbf{x}(\overline{v})$, which is a contradiction. 
In conclusion, the replace operation can only be applied to nodes in $T$ that do not correspond to nodes in $ \overline{N}$.

We have $|\overline{N}|\ge \sfrac{(2^n-1)}{n}$, so $T$ cannot be compressed to fewer than $\sfrac{(2^n-1)}{n}$ nodes, which proves~\ref{item3}.
\qed
\end{proof}

We believe an interesting direction in future work is to better understand the following high-level question: if a BB tree is generated using directions from a set $\mcf{D}$, then can it be compressed using allowable directions from a different set $\mcf{D}'$?
We briefly explore this question when $\mcf{D} = \mcf{D}'$ and the BB tree is created using strong branching with best bound selection.
The conclusion we make is that the drop operation is sufficient for compressing these BB trees.
In other words, the strong branching rule is strong enough to render the replace operation ineffective beyond the drop operation.

We say that a BB tree $T$ is built using {\it full strong branching with best bound selection based on $\mcf{D} \subseteq \mbb{Z}^n$} if $T$ is constructed in the following manner:
When branching, we apply a disjunction $(\mbfs{\pi}^\top \mbf{x} \le \pi_0) \lor (\mbfs{\pi}^\top \mbf{x} \ge \pi_0+1)$ to a leaf node $v \in L(T)$ that provides the greatest increase in dual bound among all tuples $(v, \mbfs{\pi}, \pi_0) \in L(T) \times \mcf{D} \times \mbb{Z}$~\cite{achterberg2005branching}.
Ties are broken arbitrarily.

\begin{proposition}\label{propStrongBranching}
Let $T$ be a BB tree generated using full strong branching with best bound selection based on $\mcf{D}\subseteq \mbb{Z}^n$.
Let $T'$ be a compression of $T$ (say $T_1 = T, T_2, \dotsc, T_k = T'$ is a sequence of BB trees that exhibits this compression) such that the following holds for all $i \in \{1, \dotsc, k-1\}$:
\[
\text{If $T_{i+1} = \replace(T_i, v,\mbfs{\pi}, \pi_0)$, then $\mbfs{\pi} \in \mcf{D}$.}
\]
That is, each replacement operation uses a direction in $\mcf{D}$.
Then the following properties hold:
\begin{enumerate}
    \item\label{SBProp1} The dual bound cannot improve during compression, i.e., $d(T, \mbf{c}) = d(T', \mbf{c})$.
    \item\label{SBProp2} There exists a compression $T''$ of $T$ with $|T''| \le |T'|$  that can be obtained using the drop operation exclusively.
\end{enumerate}
\end{proposition}

\begin{proof}
As a first case, suppose that $T' = \replace(T, v, \mbfs{\pi}, \pi_0)$.
For each polyhedron $\mcf{R}\subseteq \mbb{R}^n$, set 
\(
{\rm val}(\mcf{R}) := \min \{\mbf{c}^\top\mbf{x}: \mbf{x} \in \mcf{R}\}.
\)
Let us first show $d(T, \mbf{c}) = d(T', \mbf{c})$. By contradiction, suppose that $d(T, \mbf{c}) < d(T', \mbf{c})$.
Let $\{v_1, v_2\}$ denote the children of $v$ in $T$, and let $\{v_3, v_4\}$ denote the children of $v$ in $T'$.
Since $T$ and $T'$ differ only in the subtree rooted at $v$ and $d(T, \mbf{c}) < d(T', \mbf{c})$, it must be that $d(T', \mbf{c}) = \min\{{\rm val}(\mcf{Q}(v_3)), {\rm val}(\mcf{Q}(v_4))\}$. Additionally, it follows that $d(T, \mbf{c}) ={\rm val}(\mcf{Q}(\hat{v}))$ for some descendant $\hat{v}$ of $v$ in $T$; this implies $\min\{{\rm val}(\mcf{Q}(v_1)), {\rm val}(\mcf{Q}(v_2))\} \le{\rm val}(\mcf{Q}(\hat{v}))$. Therefore,
\begin{align*}
\min\{{\rm val}(\mcf{Q}(v_1)), {\rm val}(\mcf{Q}(v_2))\} &\le {\rm val}(\mcf{Q}(\hat{v}))&\\
&= d(T, \mbf{c})\\
&< d(T', \mbf{c})\\
&= \min\{{\rm val}(\mcf{Q}(v_3)), {\rm val}(\mcf{Q}(v_4))\}.
\end{align*}
On the other hand, since $v_1$ and $v_2$ were created using full strong branching, we have $\min\{{\rm val}(\mcf{Q}(v_3)), {\rm val}(\mcf{Q}(v_4))\} \le \min\{{\rm val}(\mcf{Q}(v_1)), {\rm val}(\mcf{Q}(v_2))\}$. 
However, this is a contradiction.
Hence, $d(T, \mbf{c}) = d(T', \mbf{c})$.

Now, let $T''$ be the BB tree obtained by applying the drop operation to $v_1$ and then $v_2$ in $T$. 
Clearly $|T''|\leq |T'|$. 
By the strong branching rule, $\min\{{\rm val}(\mcf{Q}(v_3)), {\rm val}(\mcf{Q}(v_4))\} \le \min\{{\rm val}(\mcf{Q}(v_1)), {\rm val}(\mcf{Q}(v_2))\}$.
Hence, we have $d(T', \mbf{c}) \leq d(T'', \mbf{c})$. 
This implies $d(T'', \mbf{c}) \geq d(T, \mbf{c})$ and thus $T''$ is indeed a compression. Note that since
the drop operation cannot strictly improve the dual bound we have $d(T'', \mbf{c}) = d(T, \mbf{c})$.

The more general case, i.e., when $T'$ is a sequence of drop operations and replacements, follows from the fact that the replacement operation, which can only be applied to non-leaf nodes, only produces new leaf nodes that either remain leaf nodes or are removed in subsequent compression operations. 
\qed
\end{proof}

When combined, the properties in Proposition~\ref{propStrongBranching} imply that only the drop operation is needed to compress BB trees when every replace operation only uses directions from $\mcf{D}$.

\section{Compression algorithms}\label{secMethods}

While the previous section provides negative results in terms of the complexity of compressing trees or even just being able to compress significantly, these are worst-case results. In practice, we may still be able to compress effectively.
In this section, we introduce two compression algorithms, an exact method and a heuristic, which we later evaluate extensively.
Henceforth, we consider the set of allowable directions to be all integer vectors, i.e.,  $\mcf{D} = \mathbb{Z}^n$.

Let $T$ be a BB tree and $\mbf{c} \in\mbb{Q}^n$.
%
For both algorithms, the general approach we follow is:
(1) Traverse $T$ starting from the root.
We may skip leaves, since these are not compressible;
%
(2) If the minimum of $\mbf{x}\mapsto\mbf{c}^\top\mbf{x}$ over $\mcf{Q}(v)$ is greater than or equal to $d(T, \mbf{c})$
then we apply $\drop(T,v)$;
%
(3) Otherwise, we search for $(\mbfs{\pi}, \pi_0) \in \mbb{Z}^n\times \mbb{Z}$ such that $T'=\replace(T,v,\mbfs{\pi}, \pi_0) $ satisfies $d(T, \mbf{c}) \geq d(T', \mbf{c})$.
%
%
In the following, we provide two methods for Step (3), which is the bottleneck of the procedure.

\subsection{An exact method} \label{sec:exactmethod}

A BB tree $\replace(T,v,\mbfs{\pi}, \pi_0)$ is a compression of $T$ if and only if 
\[ \min \{\mbf{c}^\top \mbf{x} :\ \mbf{x} \in \mcf{Q}(v),\, \mbfs{\pi}^\top \mbfs{x} \le \pi_0\} \geq d(T, \mbf{c}) \]
and
\[\min \{\mbf{c}^\top \mbf{x} :\ \mbf{x} \in \mcf{Q}(v),\, \mbfs{\pi}^\top \mbfs{x} \ge \pi_0 + 1 \} \geq d(T, \mbf{c}).
\]

Mahajan and Ralphs \cite{MR2009} propose a MIP formulation that can be used to find such $(\mbfs{\pi}, \pi_0)$; the main difference between their work and ours is that they used the MIP to find a general disjunction that could provide the best possible dual improvement when branching, but we can easily adapt it to our compression task.
The resulting model we use is
\begin{equation}\label{MRmodel}
\max_{\substack{\delta, \mbf{p}, \mbf{q}, \mbfs{\pi},\\ \pi_0, s_L, s_R }} \left\{ \delta :~~
\begin{array}{ll}
\mbf{A}^\top \mbf{p} - s_L \mbf{c} - \mbfs{\pi} = \mbf{0},
& \mbf{p}^\top \mbf{b} -  d(T, \mbf{c}) s_L - \pi_0 \geq \delta \\[.05 cm]
\mbf{A}^\top \mbf{q} - s_R \mbf{c} + \mbfs{\pi} = \mbf{0},
& \mbf{q}^\top \mbf{b} -  d(T, \mbf{c}) s_R - \pi_0 \geq \delta - 1 \\[.05 cm]
\mbf{p}, \mbf{q} \ge \mbf{0},\ s_L, s_R  \geq 0,& \mbfs{\pi} \in \mathbb{Z}^n,\ \pi_0 \in \mathbb{Z}
\end{array}
\right\}
\end{equation}
Any feasible solution with $\delta > 0$ produces a tuple $(\mbfs{\pi}, \pi_0)$ that we can use in the replace operation. 
Conversely, if no such $\delta$ exists, neither does a suitable disjunction; see \cite{MR2009}.
We note that in \cite{MR2009}, the authors fix $\delta$ to be a small constant and deal with a feasibility problem. In our case, we opted for an optimization version with $\delta$ variable.

Model \eqref{MRmodel} can be costly to solve in practice.
However, if given enough time, one can be certain that it will yield an algorithm capable of compressing $T$ as much as
possible; this will be of great use as a benchmark of compressibility.

\begin{algorithm}[t]
	\SetAlgoLined
	\text{\bf Input:}  $\mbf{A} \in \mbb{Q}^{m\times n}$ and $\mbf{b} \in \mbb{Q}^m$ defining a polytope \(\mcf{P} := \left\{\mbf{x} \in \mbb{R}^n:\ \mbf{A} \mbf{x} \le \mbf{b}\right\}\), an objective function $\mbf{c}$ and a BB tree $T$\;
	Let $L$ be a queue with the nodes of $T$\; \label{step:queue}
	\While{$L\neq \emptyset$}{
            $v = \pop(L)$\; 
            $T' = \drop(T,v)$\;
            \If{$|T'| < |T| \land d(T',\mbf{c}) \geq d(T,\mbf{c})$}{
            Remove all $w\in T\setminus T'$ from $L$\;
            $T\leftarrow T'$\; 
            \mbox{continue}\;}
            Solve problem \eqref{MRmodel} and obtain an optimal solution \((\delta, \mbf{p}, \mbf{q}, \mbfs{\pi}, \pi_0, s_L, s_R)\)\; \label{step:opt}
		\If{$\delta > 0$}{
            $T' =\replace(T,v,\mbfs{\pi}, \pi_0)$\;
            Remove all $w\in T\setminus T'$ from $L$\;
            $T\leftarrow T'$\; \label{step:last}}
	}
	\vskip .1cm
	\KwResult{A compressed tree $T$}
	\caption{Exact compression algorithm}\label{alg:exact}
\end{algorithm}

In Algorithm \ref{alg:exact}, we formalize the exact compressibility method. We remark that, in practice, Step \ref{step:opt} is not necessarily solved to optimality, and an early stopping criterion can be implemented based on the value of $\delta$. In addition, we purposely leave the ordering for the node processing (Step \ref{step:queue}) to be ambiguous. 
Below we will specify different options to test in practice; these different options do not affect the correctness of the algorithm.

\subsection{A heuristic method\label{subsec:heur}}

As mentioned before, solving problem \eqref{MRmodel} can be costly; in some instances, even finding a solution with $\delta > 0$ is impractical. To alleviate this computational burden, we explore how to replace this optimization problem with a heuristic that can efficiently find a branching direction to be used in the replace operation.

Many heuristic methods for finding good branching directions have been proposed in the literature (e.g. ~\cite{CLN2011,GM2022,KC2011,OM2001}) and can be readily used for tree compression.
Here, we adapt a procedure in Owen and Mehrota \cite{OM2001} that iteratively improves variable directions by changing one coefficient at a time.

Our heuristic compression method follows Algorithm \ref{alg:exact}, but instead of executing steps \ref{step:opt}-\ref{step:last}, it performs the following.
Assume we have solved the LP relaxation at a node $v$.
The first step is to find the best variable direction $\mbfs{\pi} \in \{ \mbf{e}^1, \dotsc, \mbf{e}^n\}$.
Suppose $\mbfs{\pi}^\top \mbf{x} \leq \pi_0$ is the side of the disjunction with the smallest optimal value.
We add this constraint to the node LP and re-solve it to obtain a fractional solution $\overline{\mbf{x}}$.
For each fractional component $\overline{x}_i$, we then evaluate the branching directions $\mbfs{\pi} + \mbf{e}^i$ and $\mbfs{\pi} - \mbf{e}^i$.
If one of these directions yields a better dual bound than $\mbfs{\pi}$, then we replace $\mbfs{\pi}$ by it and repeat the procedure until $\mbfs{\pi}$ can no longer be improved.
At the end, if the bound provided by $\mbfs{\pi}$ is better than the tree bound, we apply $\replace(T, v, \mbfs{\pi}, \pi_0)$. 
We refer the reader to \cite{OM2001} for more details on this disjunction-finding procedure.

Unlike the exact method presented in Subsection~\ref{sec:exactmethod}, this iterative method provides no guarantees that a suitable disjunction will be found, even if it exists, and therefore may not achieve the best compression.
However, the iterative method is typically much faster.

\section{Computational experiments on MIPLIB 3.0 \label{secExperiments}}

In this section, we attempt to compress MIPLIB 3.0 trees using the methods described in the previous section.
Our main goal is to evaluate, without taking running time into consideration, how compressible are realistic BB trees generated by two commonly-used branching strategies --- \emph{full strong branching (FSB)} and \emph{reliability branching with plunging (RB)}.
Our secondary goal is to estimate how much of this compression can be achieved in shorter and more practical running times.
For these experiments, we chose MIPLIB 3.0, so that we could compute large FSB trees for all instances and could obtain accurate results for the exact compression method. 
This allows us to have a point of comparison for the more practical method given in Section \ref{subsec:heur}.
We consider more challenging instances in Section \ref{sec:miplib2017}.

\subsection{Methodology} \label{sec:methodology_miplib3}

For each branching strategy and for each MIPLIB 3.0 instance, we started by generating a BB tree using a custom textbook implementation of the BB method.
We used a custom implementation of the BB method, instead of exporting the tree generated by a commercial MIP solver, so that we could easily understand how exactly the tree is generated and control every aspect of the algorithm.
The implementation is written in Julia 1.8 and has been made publicly available as part of the open-source MIPLearn software package \cite{MIPLearn}.
It relies on an external LP solver, accessed through JuMP~\cite{DunningHuchetteLubin2017} and MathOptInterface~\cite{legat2021mathoptinterface}, to solve the LP relaxation of each BB node and to evaluate strong branching decisions.
In our experiments, we used Gurobi 9.5~\cite{Gurobi} with default settings as the LP solver.
When generating the trees, we provided the optimal value to the BB method and imposed a 10,000-node limit.
No time limit was imposed, and no presolve or cutting planes were applied.

After the trees were generated, they were then compressed by the exact and the heuristic methods described in Section~\ref{secMethods}.
Both methods were implemented in Python 3.10 and {\tt gurobipy}.
The nodes were traversed using depth-first search.
For the exact method, we imposed a 24-hour limit on the entire procedure and a 20-minute limit on each individual MIP.
For the heuristic method, we imposed a 15-minute limit on the entire procedure and no time limits on individual nodes.
All MIPs and LPs were solved with Gurobi 9.5 with default settings.
The experiments were run on a dedicated desktop computer (AMD Ryzen 9 7950X, 4.5/5.7 GHz, 16 cores, 32 threads, 128 GB DDR5), and 32 trees were compressed in parallel at a time; each compression was single-threaded.

\subsection{Full strong branching results \label{subsec:fsb}}

As described in Section \ref{secTheory}, \emph{full strong branching} (FSB) is a strategy that solves, at each node of the BB tree, two LPs per potential disjunction. Here, we consider the case of trees created using variable disjunction only, thus, this strategy solves two LPs for each fractional variable. Then, it picks the branching variable that presents the best overall improvement to dual bound~\cite{achterberg2005branching}.

FSB is often paired, as we do in our experiments, with \emph{best-bound node selection} rule, which always picks, as the BB node to process next, an unexplored leaf node that has minimal optimal value.
Although computationally expensive, FSB is typically considered the state-of-the-art branching strategy in terms of node count. Furthermore, as we showed in Proposition \ref{propStrongBranching}, these are trees that we can expect to be hard to compress unless the compression procedure considers more disjunctions than the ones used in the creation of the tree.

\begin{figure}[t]
    \includegraphics[width=\textwidth]{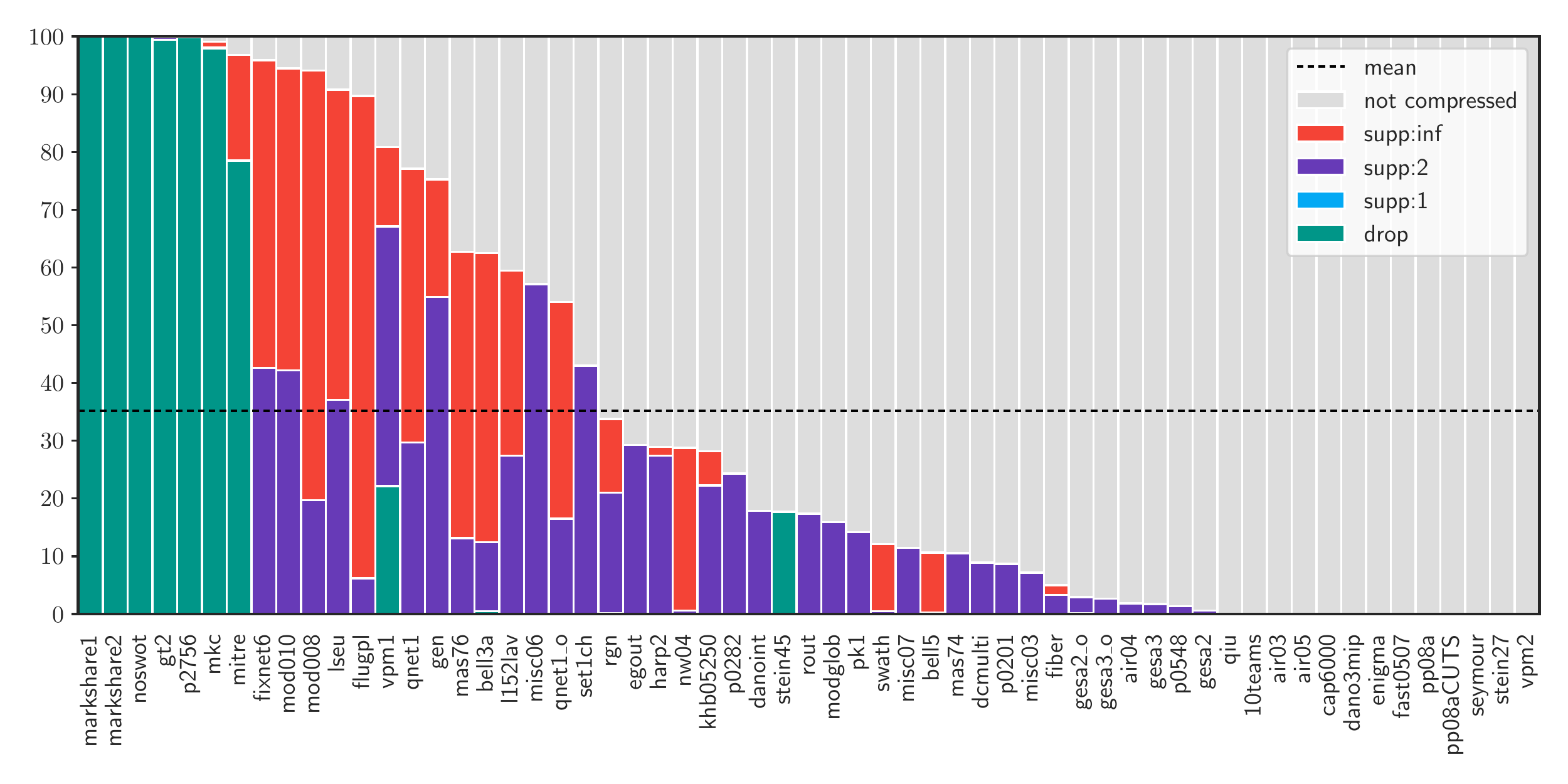}
    \caption{Compressibility of FSB trees (exact method, 24-hour limit). \label{fig:fsb-exact}}
\end{figure}

Figure~\ref{fig:fsb-exact} shows the exact compressibility of FSB trees under different restrictions on the support size of the disjunction.
Specifically, {\tt supp:inf} corresponds to the exact method based on Model~\eqref{MRmodel}, whereas {\tt supp:1} and {\tt supp:2} use the same model, but impose the additional constraint that at most 1 or 2 coefficients of $\mbfs{\pi}$, respectively, can be non-zero.
Method {\tt drop} is the method in which we are only allowed to drop nodes, not replace them.
In the chart, the compressibility of different methods is superimposed, with the weaker methods in the foreground and the stronger methods in the background. The y-axis indicates how small is the resulting tree, with larger values indicating higher compression.
For example, on instance {\tt vmp1}, methods {\tt drop}, {\tt supp:2} and {\tt supp:inf} were able to reduce the tree by 22.2\%, 67.1\% and 80.9\%, respectively.
Method {\tt supp:1} is not visible in the chart because it was not able to improve upon {\tt drop}.
%
%
The line shows the average compression obtained by the strongest method across all instances.

Our first insight from Figure~\ref{fig:fsb-exact} is that many FSB trees can be significantly compressed, despite the notorious tree-size efficiency of this branching rule.
On average, {\tt supp:inf} was able to reduce tree size by 35.2\%, with the ratio exceeding 50\% for 20 (out of 59) instances.
We also note, from the figure, that a large support size is required for obtaining the best results, although a restricted support size still provides significant compression.
On average, {\tt supp:2} compressed the trees by 24.0\%, which is still considerable, although being well below {\tt supp:inf}.
Method {\tt supp:1}, on the other hand, never outperformed {\tt drop};
this was expected in light of Proposition \ref{propStrongBranching}.
Also as a direct consequence of using the best-bound node selection rule, we observed that, for the vast majority of instances, few nodes could be dropped.
On average, {\tt drop} was only able to compress the trees by 12.1\% on average, with the compression being near zero for 50 instances.
Finally, despite the positive average compression results for {\tt supp:inf}, we note that a large number of trees could not be meaningfully compressed.
Specifically, {\tt supp:inf} presented a compression ratio below 5\% for 19 instances, which may indicate that trees for certain classes of problems are hard to compress.
Furthermore, {\tt supp:inf} took an exceedingly long average time of 47,153 seconds, with 25 instances hitting the 24-hour limit.\\

We now focus on more practical tree compression algorithms.
Figure~\ref{fig:fsb-heur} shows the performance of the heuristic method, outlined in Subsection~\ref{subsec:heur}, on the same BB trees, with a 15-minute limit.
We see that the heuristic method is able to obtain compression ratios comparable to {\tt supp:inf} in a reasonable amount of time.
On average, {\tt heuristic} took 493 seconds to run (95x faster than the exact method), and reduced tree size by 27.7\% (7.5 percentage points lower).
We conclude that FSB trees are compressible not only in a theoretical sense, but also in practice.
We also note that {\tt heuristic} outperformed {\tt supp:inf} for 12 instances, sometimes by a significant margin.
Notable examples include instances {\tt bell5}, {\tt bell3a}, {\tt vpm2}, {\tt p0282} and {\tt mas74}, where the margin exceeded 15 percentage points.
This is possible due to the time limits imposed on {\tt supp:inf}.

\begin{figure}[t]
    \resizebox{\textwidth}{!}{\includegraphics{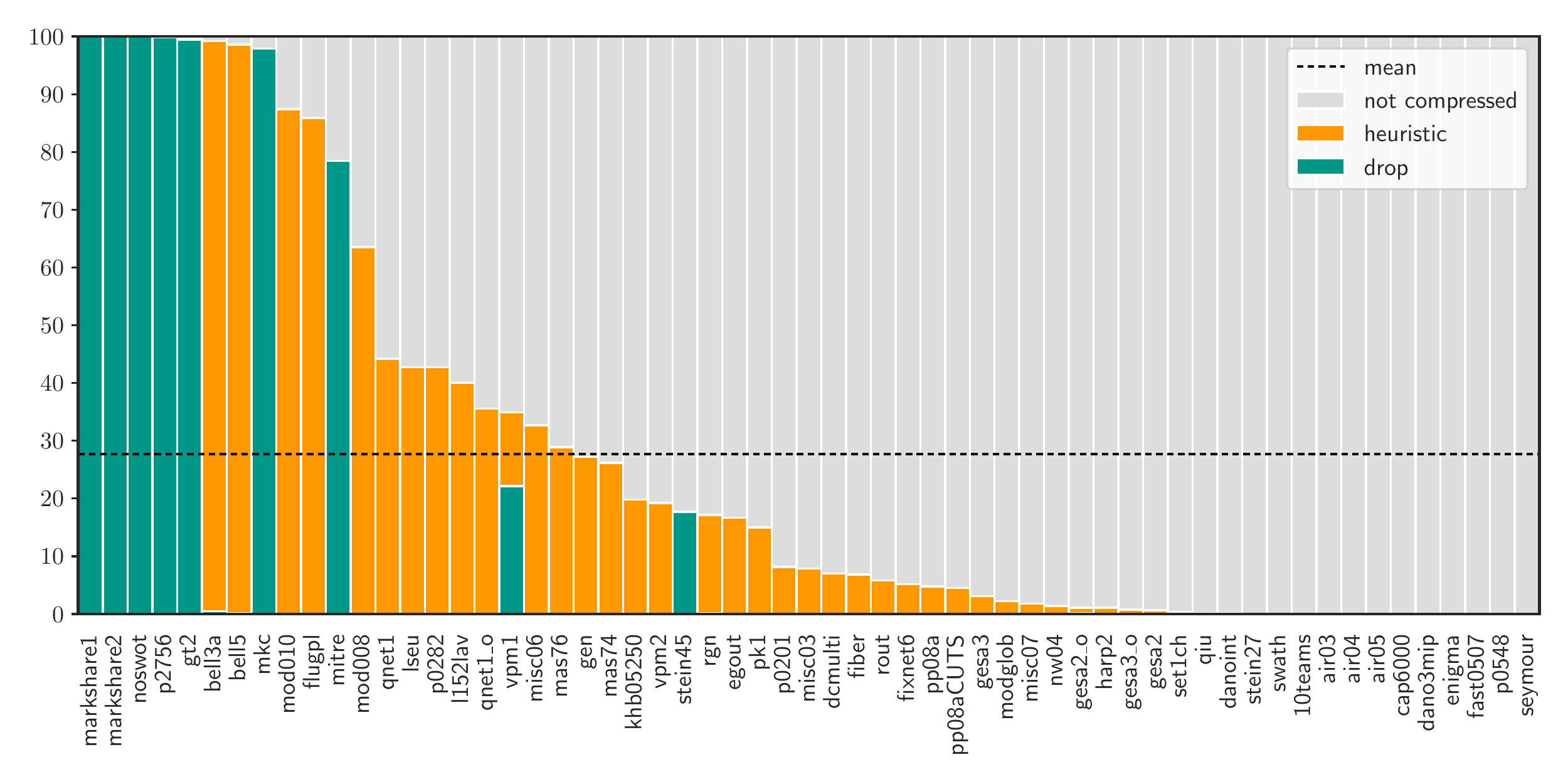}}
    \caption{Compressibility of FSB trees (heuristic method, 15-min limit). \label{fig:fsb-heur}}
\end{figure}

\subsection{Reliability branching with plunging}

\emph{Reliability branching} (RB) is a strategy that attempts to accelerate FSB by skipping strong branching computations for variables that already have reliable pseudocosts~\cite{achterberg2005branching}.
In our experiments, the pseudocost of a variable is considered reliable if it is based on 10 or more strong branching evaluations.
RB has been shown to perform well on a variety of real-world instances and it is often considered the state-of-the-art branching rule in terms of running time.
\emph{Plunging} is a modification to node selection which attempts to exploit the fact that sequentially solving two LPs that are similar can done much faster than solving two LPs that are significantly different.
When plunging is enabled, the BB method picks, as the node to explore next, one of the children of the most-recently explored node, falling back to \emph{best-bound node selection} when both children are pruned.
Our motivation for analyzing RB trees with plunging is that we expect such trees to resemble the ones generated by typical state-of-the-art MIP solvers.

Figure~\ref{fig:rbp-exact} shows the exact compressibility of RB trees, under different support size restrictions.
The first notable fact is that RB trees are, as expected, much more compressible than FSB trees.
On average, {\tt drop}, {\tt supp:1}, {\tt supp:2} and {\tt supp:inf} were able to reduce tree size by 51.9\%, 57.3\%, 61.5\% and 66.3\%, respectively.
Method {\tt supp:inf} presented compression ratio above 50\% for 42 (out of 59) instances, and above 80\% for 34 instances.
The strong performance of {\tt drop} can be directly attributed to plunging.
While the technique may be helpful when solving MIPs, we observed that it leads to the exploration of areas in the tree that do not contribute to its overall dual bound, and which can be dropped in a post-hoc analysis.
As with previous experiments, the best compression results were obtained with larger support sizes, although, in this case, the benefits of unbounded support were not as large as before, in relative terms.
Method {\tt supp:1}, unlike in previous experiments, provided significant compression in a number of instances (e.g. {\tt gen}, {\tt l152lav}, {\tt qnet1\_o}), and a modest average improvement over {\tt drop}.
We attribute this to suboptimal variable branching decisions made by RB, which is also expected.
As in the previous case, we note that {\tt supp:inf} failed to meaningfully compress a few instances, and it was overall prohibitively slow, requiring 45,256 seconds on average.

\begin{figure}[t]
    \resizebox{\textwidth}{!}{\includegraphics{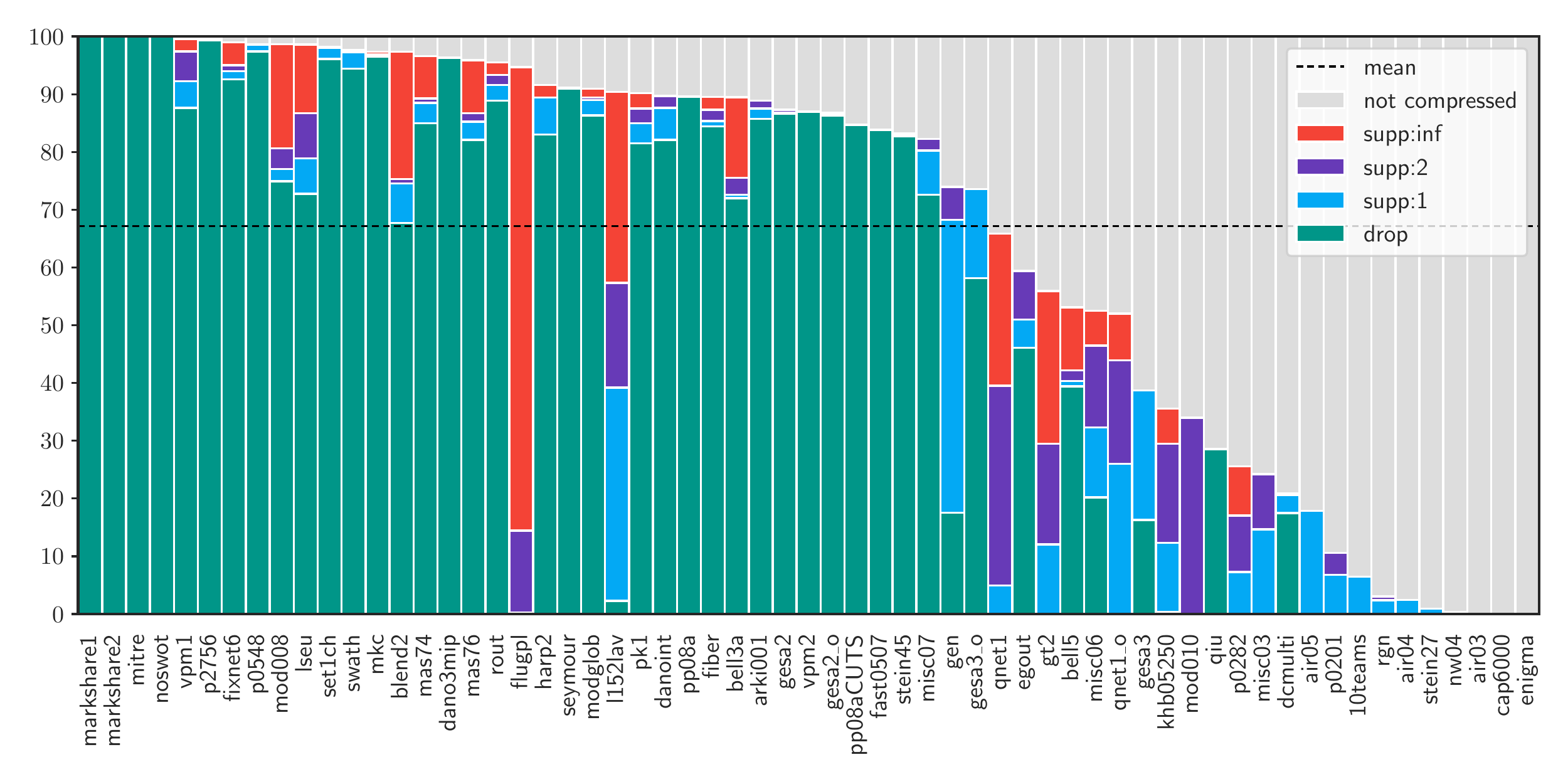}}
    \caption{Compressibility of RB trees (exact method, 24-hour limit). \label{fig:rbp-exact}}
\end{figure}

Finally, Figure~\ref{fig:rbp-heur} shows the performance of the heuristic method on RB trees.
Similarly to the results in the previous section, the heuristic method presented very strong performance, obtaining compression ratios that approached or even exceed those of the exact method, in much smaller running times.
Method {\tt heuristic} took an average of 335 seconds (134x faster) and obtained an average compression ratio of 63.7\% (2.5 percentage points lower).
We conclude that BB trees generated by node and variable selection rules that focus on MIP solution time tend to be highly compressible, in both a theoretical and a practical sense.

\begin{figure}[t]
    \resizebox{\textwidth}{!}{\includegraphics{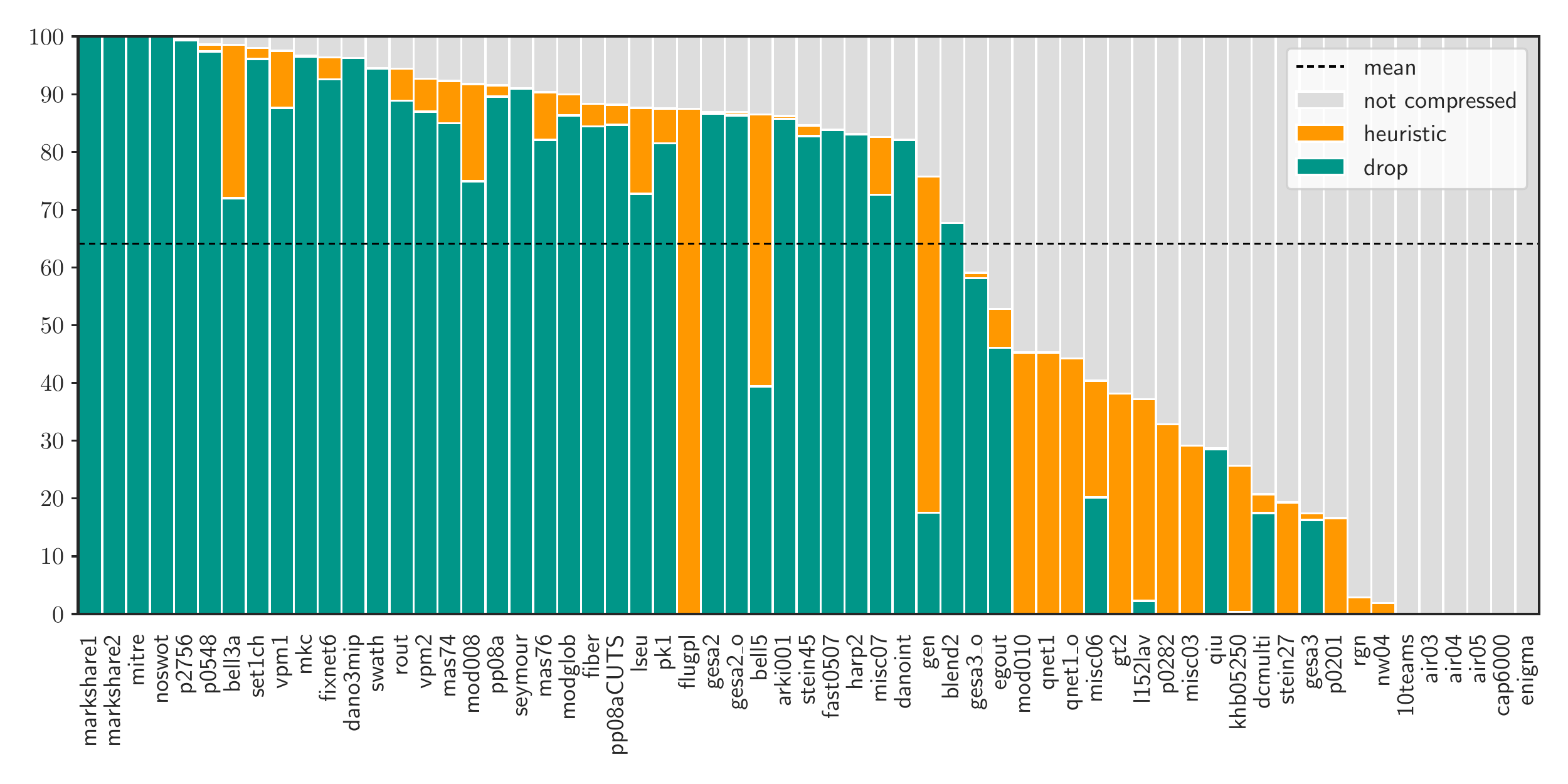}}
    \caption{Compressibility of RB trees (heuristic method, 15-min limit). \label{fig:rbp-heur}}
\end{figure}


\section{Computational experiments on MIPLIB 2017 \label{sec:miplib2017}}

In the previous section, we saw that BB trees generated by commonly-used branching rules are highly compressible, and that simple heuristic methods can obtain strong compression ratios in reasonable running times for relatively small-sized benchmark instances.
Here, we consider significantly larger instances, taken from the MIPLIB 2017 benchmark~\cite{miplib2017}.
For these instances, determining if a subtree at a given node is compressible becomes significantly more expensive, and visiting the nodes in depth-first order, as we did in the previous section, proves to be suboptimal, making the heuristic method take prohibitively long to run.

In this section, in an attempt to make the heuristic method practical for large instances, we explore different and potentially more efficient node orderings.
Since \emph{full strong branching} (FSB) is prohibitively expensive on MIPLIB 2017 instances, we compress \emph{reliability branching} (RB) trees without plunging.
For similar reasons, we omit the exact method from the analysis.
In Subsection~\ref{subsec:miplib2017-methodology}, we describe the experimental methodology, node orderings considered and evaluation metrics.
In Subsection~\ref{subsec:miplib2017-results}, we present the results.

\subsection{Methodology}
\label{subsec:miplib2017-methodology}

To generate the BB trees for the MIPLIB 2017 instances, we used the same hardware environment, programming language and computational tools as in the previous section.
To evaluate multiple node orderings fairly and effectively, we precomputed compressibility information for all trees.
Specifically, for each node of each tree, we ran the heuristic method described in Subsection~\ref{subsec:heur} with a 20-minute limit to determine whether the node is compressible, then stored the result along with the time required for the method to run.
Once this information was collected, we were able to quickly simulate the execution of the heuristic compression algorithm under different node orderings and global time limits.
Note that this methodology implicitly assumes that node processing time is deterministic, regardless of the ordering.
This was intentional, as it allows us to avoid  performance variability issues and focus solely on the effect of the ordering.
We compared six different ordering strategies, described below.

\medskip

\begin{tabular}{@{\hskip -0.35 cm}ll}
    \texttt{Random} & Nodes are visited randomly.\\[.1 cm]
    \texttt{DFS} & Nodes are visited in a depth-first search sequence, as in\\
    &Section \ref{secExperiments}.\\[.1 cm]
    \texttt{NodeId} & Nodes are visited according to their sequential numerical\\
    &identifier. Here the last created nodes are visited first, and\\
    &the root is visited last.\\[.1 cm]
    \texttt{SubtreeSize} & Nodes with smaller subtrees are visited first.\\[.1 cm]
    \texttt{Gap}& Nodes with LP objective value close to the global tree\\
    &bound are visited first. Specifically, nodes are sorted\\
    &according to \(|\text{GlobalBound} - \text{NodeBound}|\).\\[.1 cm]
    \texttt{Expert} & Nodes that are compressible are processed first, and ties are\\
    &broken by $\text{SubtreeSize}/\text{ProcessingTime}$, with higher values\\
    &being processed first.
\end{tabular}

\medskip


The \texttt{Random} strategy is included as a baseline, while \text{DFS} is included to make the results in this section comparable to the ones presented previously.
The \texttt{NodeId} and \texttt{SubtreeSize} strategies are included to test the idea that nodes created later in the tree may be more easily processed, because they have a smaller feasible region, although compressing them may not lead to a significant reduction in tree size.
The \texttt{Gap} strategy exploits the hypothesis that nodes with objective values close to the global bound are likely to be compressible because the new disjunction only needs to be slightly stronger than the current one to be useful.
Finally, \texttt{Expert} provides us an estimate of how far from the theoretical best the other strategies are.
 
We stress that \texttt{Expert} cannot be implemented in practice, as it requires pre-computed knowledge of the compressibility of every node, which is not available in what would be a realistic execution of a compression algorithm.
However, it provides a meaningful point of comparison.
Breaking ties according to $\text{SubtreeSize}/\text{ProcessingTime}$ is akin to the well-known LP solution to a knapsack problem, where items are sorted according to the ratio between their profit and their weight, and then greedily added to the knapsack.
Here, the knapsack capacity would be the global time limit, item weights would be node processing times, and item profits would be the compression potential of a node (subtree size).
\texttt{Expert} therefore is an optimal node ordering strategy for a given compressibility information.

To evaluate the effectiveness of each ordering, we consider two metrics.
First, as in the previous section, we consider the compression ratio after a given time limit $T$.
Second, to capture the evolution of the tree size over time more accurately, we compute the \emph{normalized area under the curve} (AUC (\%)), given by
\[
    \text{AUC (\%)} = 100 \times \sum_{i=1}^n \frac{t_i {s_{i-1}}}{Ts_0},
\]
where $n$ is the number of processed nodes, $t_i$ is the processing time of the $i$-th node, $s_i$ is the size of the tree after processing the $i$-th node, $s_0$ is the original tree size, and $T$ is the global time limit.
Note that AUC (\%) $\in [0,100]$, with lower values indicating better performance.
For instance, if only one node is processed during the execution of the algorithm, then $t_1 = T$ and $\text{AUC} = 100$, which indicates that the tree remained the same size during the complete execution of the algorithm.

\subsection{Results}
\label{subsec:miplib2017-results}

\begin{table}[t]
    \begin{center}
    \caption{
        Average compression ratio and normalized AUC obtained under different node orderings and time limits for RB trees (without plunging) on MIPLIB 2017 instances.
    }
    \label{table:miplib2017-summary}
    \begin{filecontents*}{miplib2017-summary.csv}
Method,AUC900,AUC3600,AUC14400,CR900,CR3600,CR14400
Expert,72.1,65.4,58.7,30.8,34.4,35.1
Gap,84.1,76.4,66.3,18.2,25.7,30.5
NodeId,86.5,79.6,69.6,15.5,21.6,27.9
SubtreeSize,87.0,79.6,69.4,15.6,21.7,28.1
Random,89.3,80.7,69.6,13.3,21.2,28.7
DFS,89.8,83.3,73.5,12.9,17.1,24.0
\end{filecontents*}
\csvreader[
        tabular=lcccc,
        head to column names,
        table head=
            \toprule \multirow{2}{*}{Node Ordering} & AUC (\%) & \multicolumn{3}{c}{Compression Ratio (\%)} \\
            \cmidrule(r){2-2} \cmidrule(l){3-5} & {1-hour} & {15-min} & {1-hour} & {4-hour} \\\midrule,
        late after line=\\,
        late after last line=\\\bottomrule,
    ]{miplib2017-summary.csv}{}
    {\texttt{\csvcoli} & \csvcoliii & \csvcolv & \csvcolvi & \csvcolvii}
    \end{center}
\end{table}

Table~\ref{table:miplib2017-summary} presents a summary of the performance of each node ordering strategy under three different time limits.
In terms of AUC (\%), we see that \texttt{DFS} is the worst performing strategy, being outperformed even by our baseline strategy \texttt{Random}.
Strategies \texttt{SubtreeSize} and \texttt{NodeId} present similar performance, slightly outperforming the baseline.
Strategy \texttt{Gap} presents the best performance among the practical strategies, being significantly better than both \texttt{SubtreeSize} and \texttt{NodeId} on average.

Comparing compression ratios under a 4-hour limit, we see that all practical strategies, even the worst performing ones, can significantly compress the BB trees, with ratios ranging from 24\% to 30.5\%.
The similarity in performance is not very surprising since, given enough time, all orderings eventually lead to the same compressed tree size.
The importance of good node orderings, however, becomes evident under shorter time limits.
With a 1-hour time limit, \texttt{Gap} provides a significant 25.7\% compression, while \texttt{DFS} only achieves 17.1\%.
Under a 15-minute limit, the results are unfortunately much worse, with even the best performing practical strategy \texttt{Gap} achieving only 18.2\% compression.

Although \texttt{Gap} is the best strategy on average, it is not the best strategy for every single instance.
The ranking of the other strategies also varies across instances.
Figure~\ref{fig:miplib2017-bar} shows a more detailed comparison of the compression ratio obtained by different node orderings under a 1-hour limit, compared to the expert ordering.
For example, in instance \texttt{radiationm18-12-05}, \texttt{Gap} provides almost no compression, while \texttt{DFS} and \texttt{Random} reduces the tree size by nearly 100\%.
Other cases in which \texttt{Gap} is significantly outperformed by other strategies include \texttt{p200x1188c} (41 percentage points below the best), \texttt{swath3} (35 p.p.), \texttt{beasleyC3} (22 p.p.), \texttt{swath1} (14 p.p.) and \texttt{neos-1456979} (12 p.p.).
These results suggest that an ensemble of node ordering strategies could be used to further improve the compression ratio.
To further illustrate this point, Figure~\ref{fig:miplib2017-time} shows the progress of the compression algorithm over time for four selected MIPLIB 2017 instances.
Besides showing that particular strategies may be better suited for particular instances, the figure also shows that some strategies may be preferrable depending on the time limit.
In \texttt{gmu-35-40}, for example, \texttt{Random} is the best strategy under a 750-second limit, but \texttt{DFS} is better for longer time limits.
Similarly, in \texttt{csched007}, strategy \texttt{NodeId} outperforms \texttt{SubtreeSize} until around 1500 seconds, then becomes clearly worse at around 3500 seconds.

Looking at the performance of \texttt{Expert} allow us to assess the potential of better orderings.
Although \texttt{Gap} provides good compression, as discussed above, there is still a significant margin between its performance and \texttt{Expert}, which indicates room for improvement.
Under a 4-hour limit, \texttt{Expert} achieves a compression ratio of 35.1\%, which is similar to the 35.2\% achieved by the exact method on MIPLIB 3 FSB trees in Subsection~\ref{subsec:fsb}.
Assuming that the compressibility of the two sets of trees is roughly similar, these results could indicate that improving the node ordering, without any other further improvements to the heuristic algorithm, might be sufficient to achieve compression levels similar to the exact method, although this hypothesis would need to be confirmed by further experiments.
The 15-minute results for \texttt{Expert} are also encouraging, showing that it is theoretically possible to achieve around 30\% compression in very short time, even for large-scale MIP problems.

\begin{figure}
    \centering
    \resizebox{\textwidth}{!}{
        \includegraphics[width=\textwidth]{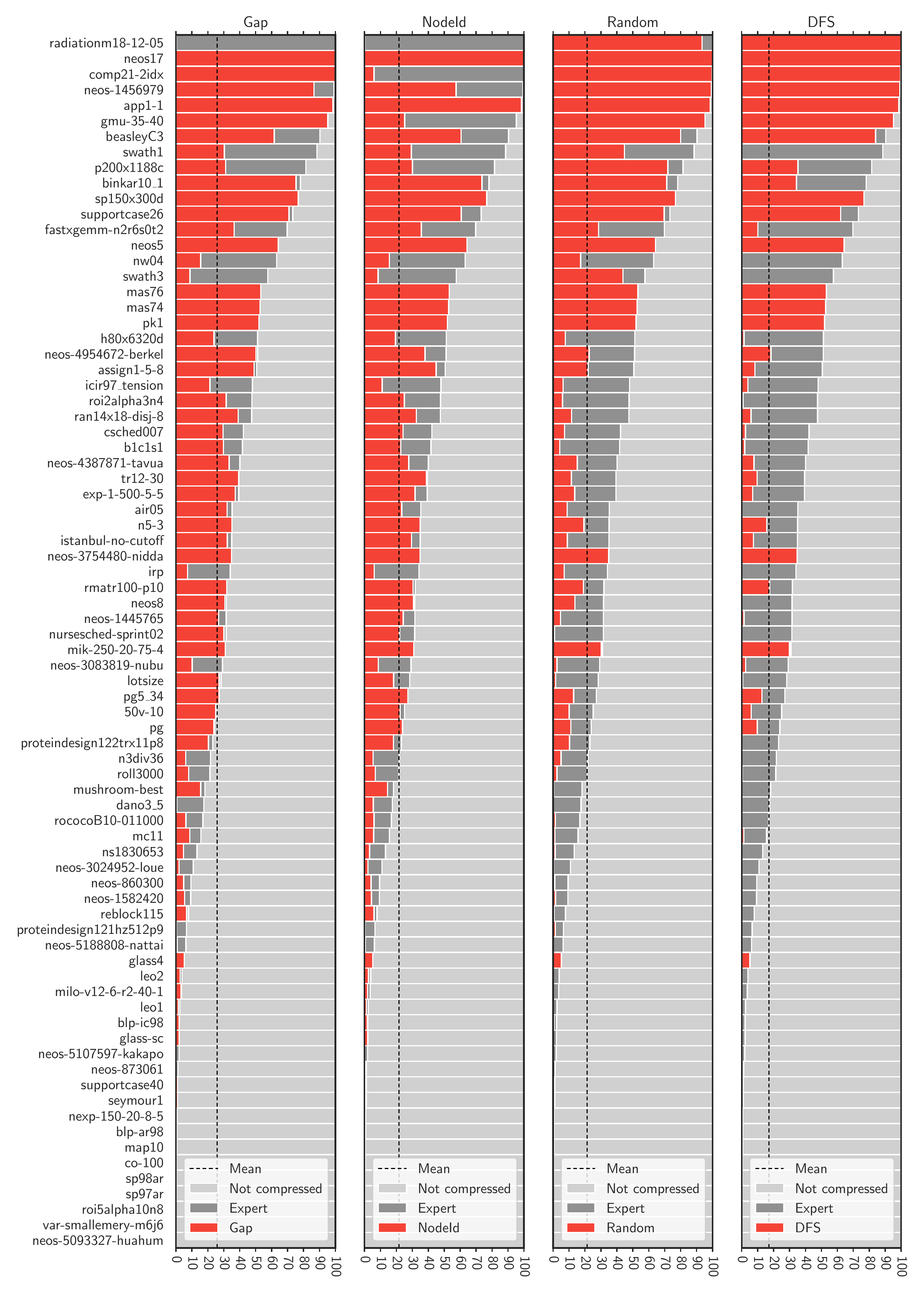}
    }
    \caption{
        Compression ratio under different node orderings and 1-hour limit for RB trees (without plunging) on MIPLIB 2017 instances, compared to the expert ordering.
        \label{fig:miplib2017-bar}
    }
\end{figure}

\begin{figure}
    \centering
    \includegraphics[scale=0.375]{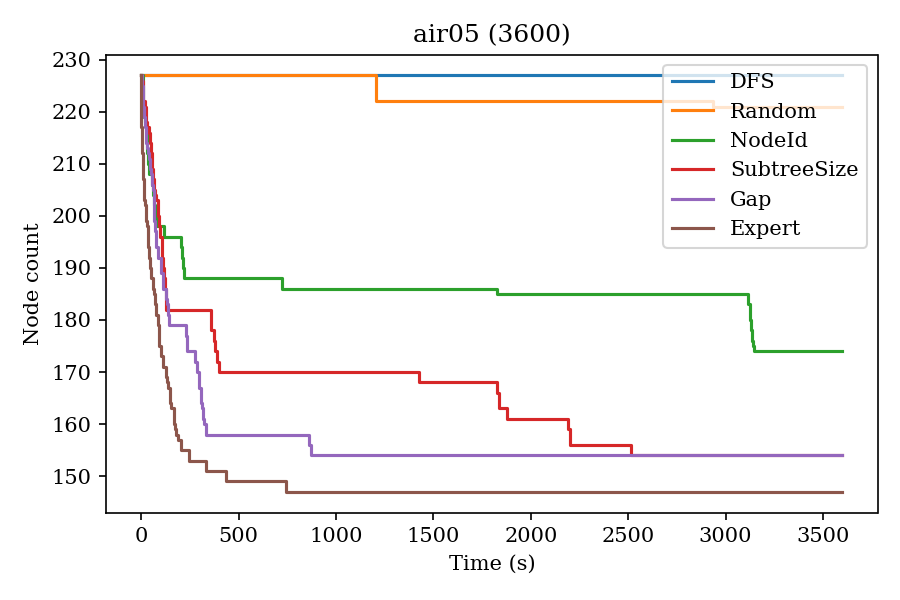}
    \includegraphics[scale=0.375]{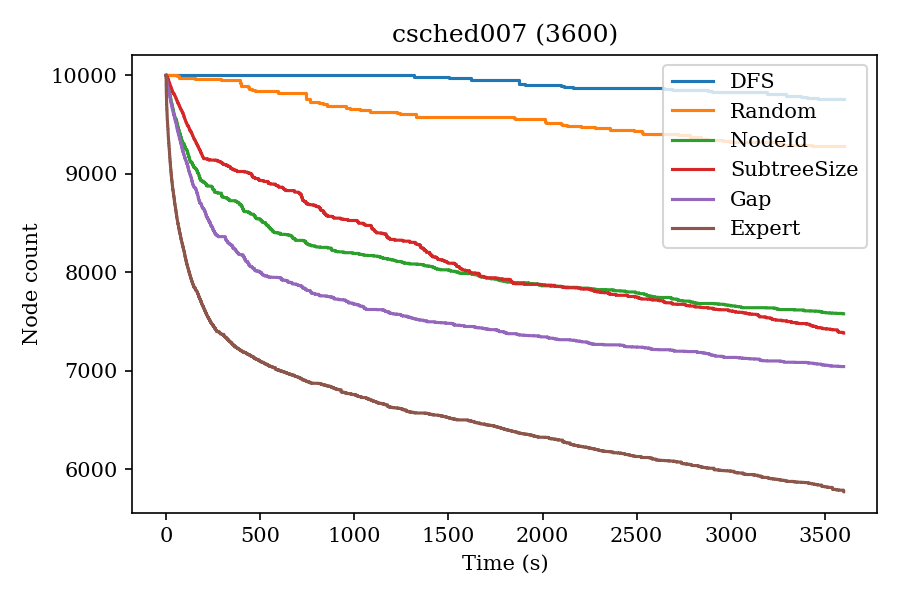}
    \includegraphics[scale=0.375]{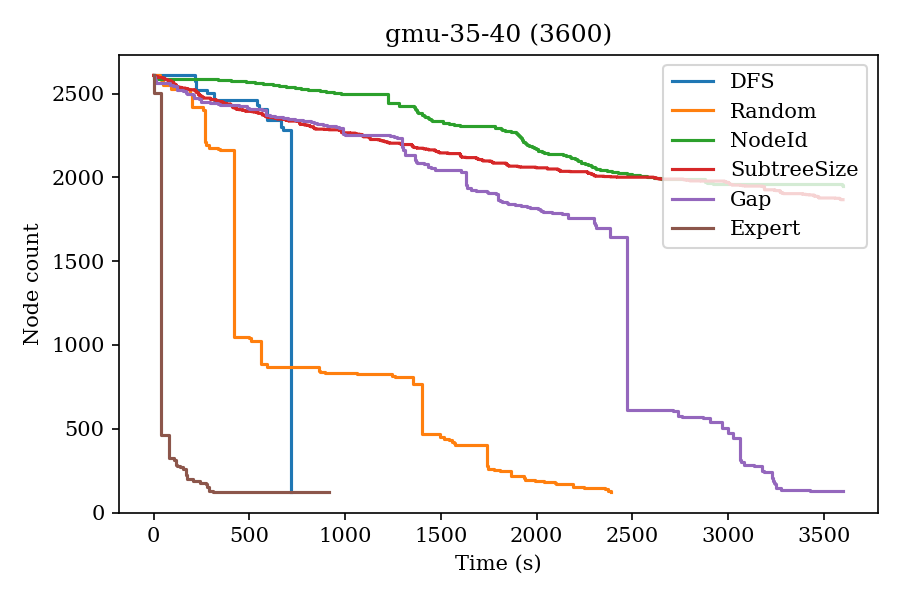}
    \includegraphics[scale=0.375]{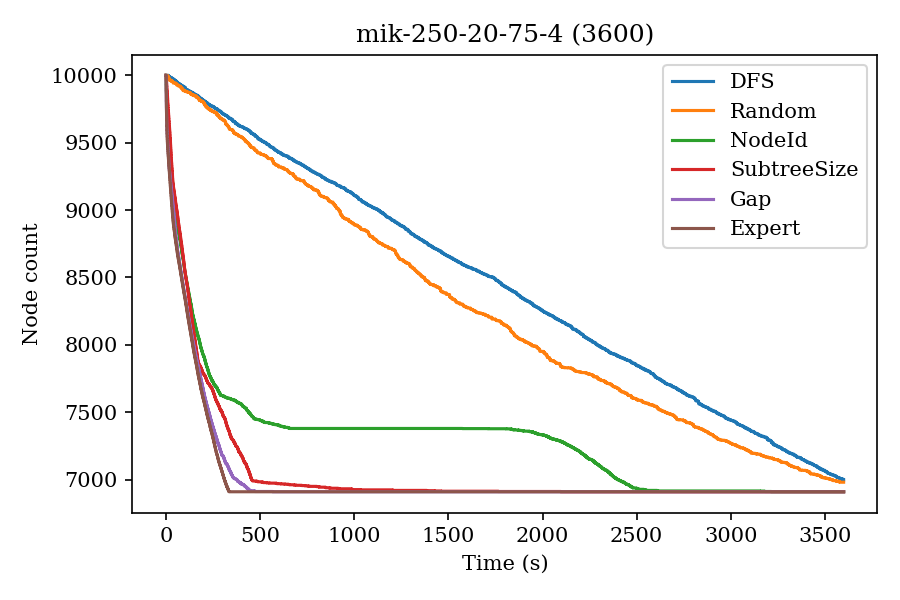}
    \caption{
        Tree size over time for four selected MIPLIB 2017 instances, under different node orderings and 1-hour limit.
    }
    \label{fig:miplib2017-time}
\end{figure}


\section{Conclusion \& Future Work}

We have formally introduced the tree compression problem, and we demonstrated through experiments how much trees can be compressed. 
There are many open questions that we believe warrant future research. 
First, is there a family of problems for which BB trees generated, say using strong branching, can be provably compressed?
Second, for a tree generated using branching directions in a set $\mcf{D}$, how compressible is the tree using directions in the Minkowski Sum $\mcf{D}+\mcf{D}$?
In particular when $\mcf{D}$ is the set of variable disjunctions, a positive result may indicate sparse disjunctions that are useful in a BB tree. 
This would complement our current computational results on disjunctions of support size $2$.
Third, given that the compression algorithm is based (partially) on general disjunctions which can be seen as splits, is there a relationship between the strength of split cuts at the root and the compressibility of a BB tree?
Finally, could the general disjunctions found by the compression algorithm be useful in solving similar MIP instances?

\medskip

\noindent{\bf Acknowledgements.} J. Paat was supported by a Natural Sciences and Engineering Research Council of Canada Discovery Grant [RGPIN-2021-02475]. \'A.S. Xavier was partially supported by the U.S. Department of Energy Office of Electricity.
%

\bibliographystyle{splncs04}
\bibliography{references}
\end{document}